\documentclass[10pt]{article}
\usepackage{amsmath,amsthm,amssymb}
 \usepackage[textwidth=3.4cm,textsize=10pt]{todonotes}

\numberwithin{equation}{section}

\newcommand{\refpart}[1]{{\it (#1)}}  

\newcommand{\PP}{\mathbb{P}}
\newcommand{\QQ}{\mathbb{Q}}
\newcommand{\CC}{\mathbb{C}}
\newcommand{\FF}{\mathbb{F}}
\newcommand{\RR}{\mathbb{R}}
\newcommand{\ZZ}{\mathbb{Z}}
\newcommand{\KG}{\Lambda}
\newcommand{\XX}{{\cal X}}
\newcommand{\KC}{{\cal K}}
\newcommand{\ia}{\alpha}
\newcommand{\ib}{\beta}
\newcommand{\PSL}{\mbox{\rm PSL}}

\newtheorem{theorem}{Theorem}[section]
\newtheorem{lemma}[theorem]{Lemma} 
\newtheorem{propose}[theorem]{Proposition} 
\newtheorem{corollary}[theorem]{Corollary} 
\newtheorem{definition}[theorem]{Definition} 
\newtheorem{remark}[theorem]{Remark} 

\newcommand{\hpg}[5]{{}_{#1}\mbox{\rm F}_{\!#2}\!
  \left(\left.{#3 \atop #4}\right| #5 \right) }

\newcommand{\hpgo}[2]{{}_{#1}\mbox{\rm F}_{\!#2}}

\title{Darboux evaluations for hypergeometric functions with the projective monodromy PSL$(2,\FF_7)$}

\author{
        Raimundas Vidunas\footnote{
        Vilnius University, Lithuania.
        E-mail: {\sf rvidunas@gmail.com}.}
       }

\date{}

\begin{document}

\maketitle

\begin{abstract}
Algebraic hypergeometric functions can be compactly expressed as radical functions
on pull-back curves where the monodromy group is simpler, say, a finite cyclic group. 
These so-called Darboux evaluations were already considered for algebraic $\hpgo21$-functions.
This article presents Darboux evaluations for the classical 
case of $\hpgo32$-functions with the projective monodromy group PSL$(2,\FF_7)$. 
As an application, appealing 
modular evaluations of the same $\hpgo32$-functions are derived.
\end{abstract}

\section{Introduction}


One way to obtain workable expressions for algebraic hypergeometric functions
is to pull-back them to algebraic curves where the (finite) monodromy group would be simpler, 
say, a finite cyclic group \cite{ViDarb}.
For example, we have
\begin{align} 
\label{eq:tetra2} \hpg{2}{1}{1/4,\;7/12}{4/3}{\frac{x\,(x+4)^3}{4(2x-1)^3}}
 & =  \frac{1}{1+\frac14x}\,\left(1-2x\right)^{3/4},\\   \label{eq:tetra3}
\hpg{2}{1}{1/2,\;5/6}{2/3}{\frac{x\,(x+2)^3}{(2x+1)^3}\,} 
& = \frac{1}{(1-x)^2}\;\big(1+2x\big)^{3/2}
\end{align}
around $x=0$.
Here the $\hpgo21$-functions have the tetrahedral group $\cong A_4$ 
as the projective monodromy group (of the hypergeometric differential equation). 
The rational arguments of degree 4 reduce the monodromy to small cyclic groups,
as evidenced by the radical (i.e., algebraic power) functions on the right-hand sides of these identities.

If a Fuchsian differential equation $E$ on the Riemann sphere $\CC\PP^1$ 
has a finite monodromy group, then $E$ can be transformed by a pull-back transformation 
with respect to an algebraic covering $\varphi:B\to\CC\PP^1$ to a Fuchsian equation 
on the curve $B$ with a cyclic monodromy. 
This is a special case of a {\em Darboux covering} as defined in \cite{ViDarb}.
The transformed equation on the  {\em Darboux curve} $B$ has a basis of radical solutions
(and hence, a completely reducible monodromy representation).
Explicit expressions for solutions of $E$ in terms of radical 
functions on $B$ are called {\em Darboux evaluations}.
In \cite{ViDarb}, all tetrahedral, octahedral and icosahedral Schwarz types \cite{Schwarz}
of algebraic $\hpgo21$-functions are exemplified by Darboux evaluations.
The simplest Darboux curves for most icosahedral Schwarz types have genus 1
rather than 0.


Algebraic generalized hypergeometric functions $\hpgo{p}{p-1}$ are classified by Beukers and Heckman
\cite{BeuHeck}. One particularly interesting case \cite{Kato11}, \cite{vdPut00} 
is algebraic $\hpgo32$-functions such that the projective monodromy group 
(of their third order Fuchsian equations) is the simple group 
\begin{equation}
\KG=\PSL(2,\FF_7)\cong \mbox{\rm GL}(3,\FF_2)
\end{equation}
with 168 elements. This is the group of holomorphic symmetries of the Klein quartic curve
\begin{equation} \label{eq:klein}
X^3Y+Y^3Z+Z^3X=0.
\end{equation}
Let us denote this curve in $\CC\PP^2$ by $\KC$.
It has the genus $g=3$. It is the simplest {\em Hurwitz curve} \cite{Macbeath} 
as it achieves the Hurwitz \cite{Hurwitz93}
upper bound $84(g-1)$ for the number of holomorphic symmetries for complex projective curves 
of genus $g>1$.

Third order Fuchsian equations with the projective monodromy group $\KG$ 
were first constructed by Halphen \cite{Halphen84} and Hurwitz \cite{Hurwitz86}.
The Hurwitz equation is satisfied by
\begin{equation} \label{eq:hurwsol}
\frac1{\sqrt{1-z}}\;\hpg32{\frac{9}{14},\frac{11}{14},\frac{15}{14}}{\frac43,\;\frac53}{z}.
\end{equation}
In \cite[Table 8.3]{BeuHeck}, classes of 
$\hpgo32$-functions with the projective monodromy group $\KG$ are labelled by the numbers 2, 3, 4.
The differential Galois group of their Fuchsian equations 
is the complex reflection group ST24 in the Shephard--Todd classification \cite{STrg},
isomorphic to the central extension $\KG\times(\ZZ/2\ZZ)$.
That is also the customary monodromy group inside ${\rm GL}(3,\CC)$.
Theorem \ref{th:3f2cases} here gives a classification up to contiguous relations
of (Fuchsian equations for) $\hpgo32$-functions with the projective monodromy group $\KG$.

Pull-back coverings 
for transformations to Fuchsian equations with smaller 
projective monodromy groups correspond to subfields of $\CC(f_2/f_1,f_3/f_1)$,
where $f_1,f_2,f_3$ is a linear basis of solutions of the original Fuchsian equation. 
The largest\footnote{Correspondingly, $(\#\KG)/7=24$ is the smallest degree 
of Darboux coverings that reduce the projective monodromy $\KG$
to a cyclic group.  But smaller degree Darboux coverings 
exist that give pull-back transformations to 
reducible monodromy representations. In \cite{Vid18b}, 
Darboux coverings of degree 21 of the considered Fuchsian equations are given. 
They transform the projective monodromy $\KG$ to a dihedral group.}  
cyclic subgroup of $\KG$ is isomorphic to $\ZZ/7\ZZ$; see the character table \cite[(1.1)]{Elkies}.
By the Galois correspondence, there are pull-back coverings
of degree $(\#\KG)/7=24$ that reduce the projective monodromy group 
to $\ZZ/7\ZZ$. The smallest degree for pull-back coverings  
to Darboux evaluations of the considered $\hpgo32$-functions is thus 24.
These pull-back coverings can be found in these ways: 
\begin{enumerate}
\item Start with a special case where a basis $f_1,f_2,f_3$ of solutions projectively generates
the function field of Klein's curve $\KC$. 
The equations of Halphen \cite{Halphen84} and Hurwitz \cite{Hurwitz86} are suitable.
The subfield of index 7 has genus 0 and is easily computable \cite[p.~67]{Elkies},
as we recall in \S \ref{sec:invkc}.
\item Alternatively, the modular curve $\XX(7)$ is isomorphic to $\KC$ 
and gives a suitable special Fuchsian equation over the $j$-line, 
while the modular curve $\XX_1(7)$ gives the subfield of index 7.
This is shown in \cite[\S 4]{Elkies}, or \S \ref{sec:modular} here.
The $j$-covering $\XX_1(7)\to\CC\PP^1$ is a Belyi map 
with the branching pattern $[7^31^3/3^8/2^{12}]$; see \cite{Hoshino10}.
\item Alternatively, $[7^31^3/3^8/2^{12}]$ is the only feasible degree 24 branching pattern
for the monodromy reduction to $\ZZ/7\ZZ$; see Proposition \ref{th:brpats}.
The only Belyi map with this branching was computed as such in 
\cite{ViAGH}.
\item Use cubic and quadratic transformations of $\hpgo32$-functions \cite{KatoTR}
to derive the degree 24 pull-back coverings for the other cases of hypergeometric functions 
with $\KG$ as the projective monodromy group.
They turn out to be genus 1 Belyi maps with the branching patterns $[7^31^3/4^6/2^{12}]$ or $[7^31^3/7^31^3/2^{12}]$.  
\end{enumerate}
The main result of this article are the degree 24 Darboux coverings $\Phi_3,\Phi_7,\Phi_4$
exhibited \S \ref{sec:c237}, \S \ref{sec:c277}, \S \ref{sec:c247} (respectively), 
and representative samples of Darboux evaluations. 
The starting map $\Phi_3$ is already known by \refpart{i}--\refpart{iii}.

As an application of computed Darboux evaluations,
we prove in \S \ref{sec:modular} the following modular evaluations of considered $\hpgo32$-functions:
\begin{align}
\label{eq:mod7a} K_1(\tau) := & \;
j(\tau)^{1/42}\;\hpg32{\!-\frac1{42},\,\frac{13}{42},\,\frac{9}{14}}{\frac47,\;\frac67}{\frac{1728}{j(\tau)}}  \\
\label{eq:mod7aa} 
= & \; q^{-1/42} \prod_{n=0}^{\infty} \frac1{(1-q^{7n+1})(1-q^{7n+2})(1-q^{7n+5})(1-q^{7n+6})}, \\[3pt]
 \label{eq:mod7b} K_2(\tau) := &\;
j(\tau)^{-5/42}\;\hpg32{\frac5{42},\,\frac{19}{42},\,\frac{11}{14}}{\frac57,\;\frac87}{\frac{1728}{j(\tau)}} \\
\label{eq:mod7bb} 
= & \; q^{5/42} \, \prod_{n=0}^{\infty} \frac1{(1-q^{7n+1})(1-q^{7n+3})(1-q^{7n+4})(1-q^{7n+6})}, \\[3pt] 
\label{eq:mod7c} K_3(\tau) := & \;
j(\tau)^{-17/42}\,\hpg32{\frac{17}{42},\,\frac{31}{42},\,\frac{15}{14}}{\frac97,\;\frac{10}7}{\frac{1728}{j(\tau)}} \\
\label{eq:mod7cc} 
= & \; q^{17/42} \prod_{n=0}^{\infty} \frac1{(1-q^{7n+2})(1-q^{7n+3})(1-q^{7n+4})(1-q^{7n+5})}.
\end{align}
Here $q=\exp(2\pi i \tau)$, and 
\begin{equation} \label{eq:jinv}
j(\tau)=\frac1q+744+196884q+21493760q^2+\ldots
\end{equation}
is the famous modular $j$-invariant.
Considering $j(\tau)$ as a function of $q$, the three identities hold as $q$-series identities.
They are analogues of known hypergeometric expressions \cite{stackex}, \cite[p.~476]{Brox}
for the Rogers-Ramanujan $q$-series \cite{BCKHSS}:
\begin{align} \label{eq:rr1a}
j(\tau)^{1/60}\;\hpg21{\!-\frac1{60},\,\frac{19}{60}}{\frac45}{\frac{1728}{j(\tau)}}
& = q^{-1/60}
\prod_{n=0}^{\infty} \frac1{(1-q^{5n+1})(1-q^{5n+4})} \\
 \label{eq:rr1b} & = q^{-1/60}
\, \sum_{n=0}^{\infty} \frac{q^{n^2}}{(1-q)\cdots(1-q^n)}, \\
 \label{eq:rr2a}
j(\tau)^{-11/60}\;\hpg21{\frac{11}{60},\,\frac{31}{60}}{\frac65}{\frac{1728}{j(\tau)}}
& =  q^{11/60} \prod_{n=0}^{\infty} \frac1{(1-q^{5n+2})(1-q^{5n+3})} \\
 \label{eq:rr2b} & =  q^{11/60}
\, \sum_{n=0}^{\infty} \frac{q^{n^2+n}}{(1-q)\cdots(1-q^n)}. 
\end{align}
The right-hand sides of (\ref{eq:rr1a})--(\ref{eq:rr1b}) and (\ref{eq:rr2a})--(\ref{eq:rr2b})
amount to the famous Rogers-Ramanujan identities \cite{RogersR}.
These $q$-series define modular functions related to the modular curve $\XX(5)$. 
Similarly, the $q$-series in (\ref{eq:mod7aa}), (\ref{eq:mod7bb}), (\ref{eq:mod7cc})
define modular functions related to the modular curve $\XX(7)$. 
They can be recognized in \cite[p.~157]{Duke}.
The $\hpgo32$-expressions in  (\ref{eq:mod7a}), (\ref{eq:mod7b}), (\ref{eq:mod7c})
are considered in \cite[Example 29]{FrancMason}.
We have the following projective parametrization of Klein's curve $\KC$: 
\begin{equation} \label{eq:modpara}
(X:Y:Z) = (-K_3(\tau):K_2(\tau):K_1(\tau));
\end{equation}
see \cite[Proposition 30]{FrancMason}, \cite[\S 5.1.1]{Kato11}.

\section{Preliminaries}

Section \ref{sec:hgp} recalls basic knowledge about differential equations for $\hpgo32$-functions, 
their pull-back transformations, and contiguous relations. 
Section \ref{sec:bhclass} classifies $\hpgo32$-functions 
with the projective monodromy $\KG={\rm PSL}(2,\FF_7)$ up to the contiguity equivalence.
Section \ref{sec:invkc} describes invariants and basic morphisms of Klein's curve $\KC$.
Section \ref{sec:darbcovs} introduces Darboux coverings 
following \cite{ViDarb}.
Section \ref{sec:ourcovs} presents the branching patterns of the considered Darboux coverings,
algebraic relations between them, and the Darboux curves of genus 1.

\subsection{Hypergeometric functions}
\label{sec:hgp}

The hypergeometric function $\hpg32{\!\ia_1,\,\ia_2,\,\ia_3\!}{\ib_1,\,\ib_2}{z}$ 
satisfies the differential equation
\begin{align} \label{eq:hpgde}
\! \left(z\frac{d}{dz}+\ia_1\!\right) \! \left(z\frac{d}{dz}+\ia_2\!\right) \! 
\left(z\frac{d}{dz}+\ia_3\!\right) \! Y(z)  \hspace{132pt} \\
=\frac{d}{dz}  \! \left(z\frac{d}{dz}+\ib_1-1\!\right) \!
\left(z\frac{d}{dz}+\ib_2-1\!\right) \! Y(z), \nonumber
\end{align}
minding the commutativity rule $\frac{d}{dz}z=z\frac{d}{dz}+1$.
This is a third order Fucshian equation with three singular points $z=0$, $z=1$, $z=\infty$.
The singularities and local exponents at them are encoded by the generalized Riemann's $P$-symbol:
\begin{equation}
P\left( \begin{array}{ccc|c} 
z=0 & z=1 & z=\infty & \\ \hline
0 & 0 & \ia_1 & \\
1-\ib_1 & 1 & \ia_2 & z \\
1-\ib_2 & \gamma & \ia_3 &
\end{array}\right) 
\end{equation}
with $\gamma=\ia_1+\ia_2+\ia_3-\ib_1-\ib_2$. 
Generically, a basis of local solutions at $z=0$ and $z=\infty$ can be written in terms
of $\hpgo32$-series. Let $M=(m_{i,j})$ denote the matrix
\begin{equation} \label{eq:hypm}
M=\left( \begin{array}{ccc} 
\ia_1 & \ia_1-\ib_2+1 & \ia_1-\ib_1+1 \\
\ia_2 & \ia_2-\ib_2+1 & \ia_2-\ib_1+1 \\
\ia_3 & \ia_3-\ib_2+1 & \ia_3-\ib_1+1
\end{array}\right).
\end{equation}
A generic basis of local solutions at $z=0$ is
\begin{align}
& \hpg32{m_{1,1\,},\,m_{2,1\,},\,m_{3,1}}{\ib_1,\;\ib_2}{z}, \quad
z^{1-\ib_2}\,\hpg32{m_{1,2\,},\,m_{2,2\,},\,m_{3,2}}{2-\ib_2,\;\ib_1-\ib_2+1}{z}, \quad\nonumber \\
& z^{1-\ib_1}\,\hpg32{m_{1,3\,},\,m_{2,3\,},\,m_{3,3}}{2-\ib_1,\;\ib_2-\ib_1+1}{z},
\end{align}
while a generic basis of local solutions at $z=\infty$ is
\begin{align}
z^{-\ia_j}\,\hpg32{\!m_{j,1\,},\;m_{j,2\,},\;m_{j,3}}{\ia_j-\ia_{k}+1,\ia_j-\ia_{\ell}+1}{z}
\end{align}
with $ j\in\{1,2,3\}$ and $\{k,\ell\}=\{1,2,3\}\setminus\{j\}$.
We refer to the set of 6 functions 
formed by 3 local hypergeometric solutions (disregarding a power factor) at $z=0$ 
and 3 such local solutions at $z=\infty$ of the same third order Fuchsian equation 
as {\em companion hypergeometric functions} to each other.

Let $B$ denote an algebraic curve. Let $\varphi(\ldots)$ denote a rational function on $B$;
it defines an algebraic covering $\varphi:B\to\CC\PP^1$.
A {\em pull-back transformation} with respect to $\varphi$ 
of a differential equation for $y(z)$ in $d/dz$ has the form
\begin{equation} \label{algtransf}
z\longmapsto\varphi(\ldots), \qquad y(z)\longmapsto
Y(\ldots)=\theta(\ldots)\,y(\varphi(\ldots)),
\end{equation}
where 
$\theta(\ldots)$ is a radical function 
on $B$.  The equations that differ by a pull-back transformation with respect to 
the trivial covering \mbox{$\varphi(\ldots)=z$} 
are called {\em projectively equivalent}.

There are several algebraic transformations for $\hpgo32$-functions \cite{KatoTR}.
Here are quadratic and cubic transformations:
\begin{align} \label{eq:t32a}
\hpg32{a,a+\frac14,a+\frac12}{b+\frac14,3a-b+1}{-\frac{4z}{(z-1)^2}} \hspace{-40pt} \\   \displaybreak
& =(1-z)^{2a}\,\hpg32{2a,2a-b+\frac34,b-a}{b+\frac14,3a-b+1}{z}\!, \qquad \nonumber \\ 
\label{eq:t32b} \hpg32{a,a+\frac13,a+\frac23}{b+\frac12,3a-b+1}{\frac{27z}{(4z-1)^3}\!} \hspace{-40pt} \\
& =(1-4z)^{3a}\,\hpg32{3a,2b-3a,3a-2b+1}{b+\frac12,3a-b+1}{z}\!, \nonumber \\ 
\label{eq:t32c} \hpg32{a,a+\frac13,a+\frac23}{b+\frac12,3a-b+1}{\frac{27z^2}{(4-z)^3}\!} \hspace{-40pt} \\
& = \left(1-\frac{z}4\right)^{\!3a}\hpg32{3a,b,3a-b+\frac12}{2b,6a-2b+1}{z}\!.  \nonumber
\end{align}
They can be understood as pull-back transformations between Fuchsian equations
(particularly, (\ref{eq:hpgde})) for hypergeometric functions \cite{ViAGH}, \cite{KatoTR}.
The involved equations have the local exponent $c=1/2$  at $z=1$, 
and the quadratic or cubic arguments $\varphi(z)$
on the left-hand sides have properly branching points in the fiber $\varphi=1$. 
This helps the number of singularities of the pulled-back equation to equal merely 3.

Two $\hpgo32$-functions whose parameters $\ia_1,\ia_2,\ia_3,\ib_1,\ib_2$ 
differ respectively by integers 
are called {\em contiguous} to each other.
This defines a  {\em contiguity equivalence} relation on the $\hpgo32$-functions.
For example, differentiating a $\hpgo32$-function gives a contiguous function, generically:
\begin{equation}
\frac{d}{dz}\hpg32{\ia_1,\ia_2,\ia_3}{\ib_1,\ib_2}{z}=\frac{\ia_1\ia_2\ia_3}{\ib_1\ib_2}\,
\hpg32{\ia_1+1,\ia_2+1,\ia_3+1}{\ib_1+1,\ib_2+1}{z},
\end{equation}
Fuchsian equations of contiguous functions have the same monodromy, generically.
For a generic set of four contiguous $\hpgo32$-functions there is a linear {\em contiguous relation} 
between them \cite{Rainville}. 
For example, differential equation (\ref{eq:hpgde}) can be rewritten as a contiguous
relation between $\hpg32{\!\ia_1+n,\,\ia_2+n,\,\ia_3+n\!}{\ib_1+n,\;\ib_2+n}{z}$ with $n\in\{0,1,2,3\}$. 
Consequently, a contiguous function to a generic $\hpgo32$-function $F$ 
can be expressed linearly in terms of $F$ and its first and second derivatives (thus, as a gauge transformation).
In particular, we have 
\begin{align}
\hpg32{\ia_1+1,\ia_2,\ia_3}{\ib_1,\ib_2}{z} = & 
\left(1+\frac{z}{\ia_1}\,\frac{d}{dz}\right) \hpg32{\ia_1,\ia_2,\ia_3}{\ib_1,\ib_2}{z},\\
\hpg32{\ia_1,\ia_2,\ia_3}{\ib_1-1,\ib_2}{z} = & \left(1+\frac{z}{\ib_1-1}\,\frac{d}{dz}\right)
\hpg32{\ia_1,\ia_2,\ia_3}{\ib_1,\ib_2}{z},
\end{align}
By linear combination, we can derive expressions like
\begin{align} \label{eq:contig11}
& (\ib_1\!-\!1)(\ia_1\!-\!\ib_2)\hpg32{\!\ia_1\!-1,\ia_2,\ia_3}{\ib_1\!-1,\ib_2}{z} = 
\Big( \ia_2\ia_3z+(\ib_1\!-\!1)(\ia_1\!-\!\ib_2) \\
& \qquad +\big((\ia_2\!+\!\ia_3\!+\!1)z+\ia_1\!-\!\ib_1\!-\!\ib_2\big)z\frac{d}{dz}+z^2(z\!-\!1)\frac{d^2}{dz^2}
\Big)\hpg32{\!\ia_1,\ia_2,\ia_3}{\ib_1,\ib_2}{z}. \quad\nonumber
\end{align}

\subsection{Fuchsian equations with the Kleinian monodromy}
\label{sec:bhclass}

Let $E$ denote a 3rd order Fuchsian equation on the Riemann sphere $\CC\PP^1$.
Suppose that $f_1,f_2,f_3$ is a basis of its solutions.
If either the (conventional) monodromy group 
or the differential Galois group \cite{vdPut00} 
of $E$ are finite, those two groups coincide with the classical Galois group of 
the finite field extension $\CC(z,f_1,f_2,f_3)\supset\CC$.
In that case, the {\em projective monodromy group} refers to
the the Galois group of the finite extension $\CC(z,f_2/f_1,f_3/f_1)\supset\CC(z)$.
Both extensions of $\CC(z)$ are Galois extensions, because the monodromy representation
gives linear transformations of $f_1,f_2,f_3$ (in GL$(3,\CC)$)
or fractional-linear transformations of $f_2/f_1,f_3/f_1$ (in PGL$(3,\CC)$).

The following proposition gives a classification of 
$\hpgo32$-functions with the projective monodromy group $\KG$,
up to the contiguity and projective equivalences, 
and up to the symmetries of companion $\hpgo32$-functions. 
\begin{propose} \label{th:3f2cases}
Up to the contiguity and projective equivalences, 
there are six classes of third order hypergeometric equations 
with the projective monodromy group $\KG$. 
Here is our notation for these classes, together with representative $\hpgo32$-solutions:
\begin{align*}
{\rm (3A)}:  & \quad \hpg32{-\frac{3}{14},\frac{1}{14},\frac9{14}}{\frac13,\;\frac23}{z}; &
{\rm (3B)}:  & \quad  \hpg32{-\frac{1}{14},\frac{3}{14},\frac5{14}}{\frac13,\;\frac23}{z}; \\
{\rm (4A)}:  & \quad  \hpg32{-\frac{3}{14},\frac{1}{14},\frac9{14}}{\frac14,\;\frac34}{z}; &
{\rm (4B)}:  & \quad \hpg32{-\frac{1}{14},\frac{3}{14},\frac5{14}}{\frac14,\;\frac34}{z}; \\
{\rm (7A)}:  & \quad \hpg32{-\frac{1}{14},\frac{1}{14},\frac5{14}}{\frac17,\;\frac57}{z}; &
{\rm (7B)}:  & \quad \hpg32{-\frac{1}{14},\frac{1}{14},\frac9{14}}{\frac27,\;\frac67}{z}.
\end{align*}
\end{propose}
\begin{proof}
The classification in \cite[\S 7]{BeuHeck} is up to the following transformations that preserve
a finite primitive monodromy group:
\begin{enumerate}
\item up to the contiguity equivalence, as integer shifts of $\ia_1,\ia_2,\ia_3,\ib_1,\ib_2$ 
are discarded by considering the parameters 
\begin{align} \label{eq:hpgg}
a_1 & = c\,\exp(2\pi i\ia_1), & a_2 & =c\,\exp(2\pi i\ia_2), & a_3 & =c\, \exp(2\pi i\ia_3), \qquad\nonumber  \\
b_1 & =c\,\exp(2\pi i \ib_1),& b_2 & = c\,\exp(2\pi i\ib_2), & b_3 & = c, 
\end{align}
of a {\em hypergeometric group} \cite[Definition 3.1]{BeuHeck};
\item up to {\em scalar shifts} \cite[Definition 5.5]{BeuHeck} affecting $c$;
they amount to projective equivalence;
\item up to interchange of the sets $\{a_1,a_2,a_3\}$, $\{b_1,b_2,b_3\}$ 
reflecting a symmetry of companion $\hpgo32$-functions; see \cite[Theorem 7.1]{BeuHeck};
\item assuming that the parameters in (\ref{eq:hpgg}) generate a cyclotomic field $\QQ(\exp(2\pi i/n))$,
up to raising those parameters to the $k$th power, with $k\in\ZZ$ coprime to $n$.
\end{enumerate}
We drop the equivalence \refpart{iv}. 
It amounts to multiplying the parameters $\ia_1,\ia_2,\ia_3,\ib_1,\ib_2$ by an integer coprime to their denominators
(and to denominators of the powers in a projective equivalence factor $\theta(\ldots)$, strictly speaking). 

In \cite[Table 8.3]{BeuHeck}, there are 3 cases of $\hpgo32$-functions 
with the projective monodromy $\KG$, numbered 2, 3, 4. 
Their indicated parameters $\ia_1,\ia_2,\ia_3,\ib_1,\ib_2$
represent our cases (3B), (4B), (7B), as the $\hpgo32$-representatives given here 
differ by the integer shift $13/14\mapsto -1/14$, or by shifting all parameters of the (7B)-representative
by $-2/7$ modulo integers. The types  (3A), (4A), (7A) are obtained by multiplying the parameters by $-1$
and making integer shifts such as $-5/14\mapsto 9/14$, $-1/3\mapsto 2/3$, etc. 
To see that equivalences \refpart{i}--\refpart{iii} do not give the same transformation,
one can check the $M$-matrices (\ref{eq:hypm}) of the starting representatives:
\begin{equation}  \label{eq:3f2mat}
\left( \begin{array}{ccc} 
\!-\frac{1}{14} & \frac{11}{42} & \frac{25}{42} \\[3pt]
\frac{3}{14} & \frac{23}{42} & \frac{37}{42} \\[3pt]
\frac{5}{14} & \frac{29}{42} & \frac{43}{42}
\end{array}\right)\!, \quad 
\left( \begin{array}{ccc} 
\!-\frac{1}{14} & \frac5{28} & \frac{19}{28} \\[3pt]
\frac{3}{14} & \frac{13}{28} & \frac{27}{28} \\[3pt]
\frac{5}{14}& \frac{17}{28} & \frac{31}{28}
\end{array}\right)\!, \quad 
\left( \begin{array}{ccc} 
\!-\frac{1}{14} &\frac{1}{14} & \frac{9}{14} \\[3pt]
\frac{1}{14} & \frac{3}{14} & \frac{11}{14} \\[3pt]
\frac{9}{14}& \frac{11}{14} & \frac{19}{14}
\end{array}\right)\!. 
\end{equation}
To check that there are no other types, one observes that multiplication of the parameters
$\ia_1,\ia_2,\ia_3,\ib_1,\ib_2$ by other odd integers modulo 14 (and coprime to 3 in the case (3B))
leads to the displayed cases modulo integer shifts.
This check is easy in the cases (3B) and (4B), as the set $\{\ib_1\!\mod\ZZ,\ib_2\!\mod\ZZ\}$ stays the same,
and each set of companion $\hpgo32$-functions contains only one equation where the denominators 
of $\ia_1,\ia_2,\ia_3$ all equal to 14. In the case (7B), 
the multiplications by $-3$ or $-5$ permute (up to the contiguity equivalence) 
the $\hpgo32$-functions in this basis of local solutions at $z=0$:
\begin{equation} \label{eq:kl7blb} \hspace{-2pt}
\hpg32{\!-\frac1{14},\,\frac1{14},\,\frac{9}{14}}{\frac27,\;\frac67}{z}\!,\,
z^{1/7}\hpg32{\!\frac1{14},\,\frac3{14},\,\frac{11}{14}}{\frac37,\;\frac87}{z}\!,\, 
z^{5/7}\hpg32{\!\frac{9}{14},\,\frac{11}{14},\,\frac{19}{14}}{\frac{11}7,\;\frac{12}7}{z}\!.
\hspace{-3pt}
\end{equation}
The multiplications by $3$ and $5$ then automatically give hypergeometric functions of type (7A).
\end{proof}
It is instructive to observe the interlacing condition \cite[Theorem 4.8]{BeuHeck}
on the 6 representative $\hpgo32$-functions.

\subsection{Invariants for Klein's curve}
\label{sec:invkc}

The $\hpgo32$-functions of type (3A) are closely related to Klein's curve $\KC$ 
and to modular functions on $\XX(7)\cong\KC$. In particular,
the hypergeometric identity \cite[Proposition 30]{FrancMason} 
\begin{align} \label{eq:kleinhpg} 
\hpg32{\frac5{42},\,\frac{19}{42},\,\frac{11}{14}}{\frac57,\;\frac87}{x}^{\!3}
\hpg32{\!-\frac1{42},\,\frac{13}{42},\,\frac{9}{14}}{\frac47,\;\frac67}{x} = \hspace{-103pt} \\ 
& \hpg32{\!-\frac1{42},\,\frac{13}{42},\,\frac{9}{14}}{\frac47,\;\frac67}{x}^{\!3}
\hpg32{\frac{17}{42},\,\frac{31}{42},\,\frac{15}{14}}{\frac97,\;\frac{10}7}{x}  \nonumber \\
& +\frac{x}{1728}\,\hpg32{\frac{17}{42},\,\frac{31}{42},\,\frac{15}{14}}{\frac97,\;\frac{10}7}{x}^{\!3}
\hpg32{\frac5{42},\,\frac{19}{42},\,\frac{11}{14}}{\frac57,\;\frac87}{x} \qquad \nonumber
\end{align}
realizes a parametrization of 
$\KC$ by solutions of a third order Fuchsian equation
that was anticipated by Klein \cite[a footnote in \S9]{Klein79},
and constructed by Halphen \cite{Halphen84}, Hurwitz \cite{Hurwitz86}.
We indicated this parametrization in (\ref{eq:modpara}).

The symmetry group $\KG=\PSL(2,\FF_7)$ of $\KC$ is realized by a 3-dimensional representation of $\KG$
acting linearly on the coordinates $X,Y,Z$ in (\ref{eq:klein}).
Here we remind invariants of this 3-dimensional representation,
present the degree 168 Galois covering $\KC\to\CC\PP^1$,
and present the quotient map of this covering by the largest cyclic subgroup $\ZZ/7\ZZ$ of $\KG$. 
This gives the degree $24=168/7$ Belyi map for Darboux evaluations
of $\hpgo32$-functions of type (3A), 
and prepares us for the modular application of Darboux evaluations in \S \ref{sec:level7}.

The 3-dimensional representation (over $\CC$) of $\KG=\PSL(2,\FF_7)$ 
has these invariants in a convenient  basis \cite[\S 6]{Klein79}, \cite[\S 1.2]{Elkies}:
\begin{align}
R_4 = & \, X^3Y+Y^3Z+Z^3X,\\  
\label{eq:inv6}
R_6 = & \, XY^5+YZ^5+ZX^5-5X^2Y^2Z^2,\\
R_{14}\! = & \, X^{14}+Y^{14}+Z^{14}
 +375\,(X^8Y^4\!Z^2+X^4Y^2\!Z^8+X^2Y^8\!Z^4)   \\ &
\! +18\,(X^7Y^7+X^7Z^7+Y^7Z^7)  
 -126\,(X^8Y^3\!Z^5+X^5Y^6\!Z^3+X^3Y^5\!Z^6)  \nonumber \\ &
\! -34(X^{11}Y^2\!Z+XY^{11}\!Z^2+X^2Y\!Z^{11}) 
 -250(X^9Y\!Z^4+X^4Y^9\!Z+XY^4\!Z^9). \nonumber 
\end{align}
The polynomial ring $\CC[R_4,R_6,R_{14}]$ is the ring of invariants 
for the extended representation of complex reflection group $\KG\times(\ZZ/2\ZZ)$ (which is ST24 in \cite{STrg}). 
The representation of $\KG$ has another invariant
\begin{equation}
R_{21}=\frac{1}{14}\,\mbox{det}\!
\left( \begin{array}{ccc}
\partial R_4/\partial X & \partial R_4/\partial Y & \partial R_4/\partial Z \\ 
\partial R_6/\partial X & \partial R_6/\partial Y & \partial R_6/\partial Z \\ 
\partial R_{14}/\partial X & \partial R_{14}/\partial Y & \partial R_{14}/\partial Z 
\end{array} \right),
\end{equation}
and there is a polynomial relation \cite[(1.17)]{Elkies}
\begin{equation} \label{eq:r21r14}
R_{21}^{\,2} \equiv R_{14}^{\,3}-1728\,R_{6}^{\,7} \quad \mbox{mod} \ R_4.
\end{equation}
As Klein's quartic curve $\KC$ is defined by $R_4=0$,
this functional congruence 
means \cite[(2.13)]{Elkies} that the map
\begin{align} \label{eq:galois168}
\Phi_0 & = \frac{1728\,R_{6}^{\,7}}{R_{14}^{\,3}} 
 =1-\frac{R_{21}^{\,2}}{R_{14}^{\,3}} 
\end{align}
defines a Galois covering $\KC\to\CC\PP^1$ 
with the Galois group $\KG$.
The covering is a Belyi map with the passport $[7^{24}/3^{56}/2^{84}]$, 
of genus 3 indeed.

The quotient covering of $\KC\to\CC\PP^1$ by $\ZZ/7\ZZ$ is given by the map \cite[(2.1)]{Elkies}
\begin{equation}
(X:Y:Z)\mapsto (a:b:c)=(X^3Y:Y^3Z:Z^3X)
\end{equation}
onto the line $a+b+c=0$ in $\CC\PP^2$. 
The quotient curve is of genus 0, therefore.
Klein's curve $\KC$ is birational to the cyclic covering $(Y/Z)^7=ab^2/c^3$. 
If we take\footnote{
We choose a different parametrization than in \cite[p.~67]{Elkies} in order to have $x_7(\tau)=O(q)$ in 
(\ref{eq:defx7}) with consistency.}
\begin{equation} \label{eq:cover7}
x=-\frac{a}{c}=-\frac{X^2Y}{Z^3}, \qquad y=-\frac{Y}{Z},
\end{equation} 
then the degree 7 cyclic covering is
\begin{equation}
y^7=x\,(x-1)^2.
\end{equation}

\subsection{Darboux coverings}
\label{sec:darbcovs}

The notions of {\em Darboux curves}, {\em Darboux coverings} 
and {\em Darboux evaluations} are introduced in \cite{ViDarb}.
The terminology is motivated by integration theory of vector fields \cite{DarbVF},
where {\em Darboux polynomials} determine invariant hypersurfaces.
In differential Galois theory \cite{Weil95}, Darboux polynomials are specified
by algebraic solutions of an associated Riccati equation.
Here is a formulation of \cite[Definition 3.1]{ViDarb}.
\begin{definition} 
Consider a linear homogeneous differential equation
\begin{equation} \label{eq:genlde}
\frac{d^n}{dz^n}+a_{n-1}(z)\frac{d^{n-1}}{dz^{n-1}}+\ldots
+a_1(z)\frac{d}{dz}+a_0(z)=0
\end{equation}
on $\CC\PP^1$, thus with $a_i(z)\in\CC(z)$. 
We say that an algebraic covering $\varphi:B\to\CC\PP^1$ is a {\em Darboux covering} for $(\ref{eq:genlde})$
if a pull-back transformation $(\ref{algtransf})$ of it with respect to $\varphi$ has a solution $Y$ such that:
\begin{enumerate}
\item[\refpart{a}] the logarithmic derivative $u=Y'/Y$ is 
a rational function on the algebraic curve $B$;
\item[\refpart{b}] the algebraic degree of $u$ over $\CC(z)$ equals the degree of $\varphi$.
\end{enumerate}
The algebraic curve $B$ is then called a {\em Darboux curve}.
\end{definition}
Condition \refpart{a} means that the monodromy representation of the pulled-back equation
has a one-dimensional invariant subspace (generated by $Y$). 
In contrast, cyclic monodromy implies a completely reducible monodromy representation,
thus a basis of $n$ radical solutions (whose logarithmic derivatives solve the Riccati equation).
This article emphasizes reduction to a cyclic projective monodromy.
As shown in \cite{Vid18b}, reduction of $\KG$ to a dihedral monodromy group 
leads to a reducible 3-dimensional monodromy representation, hence a Darboux evaluation.

Determination of Darboux coverings is made easier by their basic properties.  
The following lemma underlines that Darboux coverings are ``invariant" 
under transformations of hypergeometric equations that preserve the monodromy.
\begin{lemma}
Let $E_1$ denote a hypergeometric equation $(\ref{eq:hpgde})$ with a finite primitive monodromy group.
Suppose that other hypergeometric equation $E_2$ is related to $E_1$ by 
transformations described in \refpart{i}, \refpart{iv} within the proof of Proposition $\ref{th:3f2cases}$.
If $\varphi:B\to\CC\PP^1$ is a Darboux covering for $E_1$,
then $\varphi$ is a Darboux covering for $E_2$ as well.
\end{lemma}
\begin{proof}
The transformations \refpart{i}, \refpart{iv} do not affect the primitive monodromy group,
thus equations $E_1,E_2$ have isomorphic monodromies.
Let $E_1^*$, $E_2^*$ denote the Fuchsian equations obtained from $E_1,E_2$, respectively,
by applying the same pull-back transformation with respect to $\varphi$.
The monodromies of $E_1^*$, $E_2^*$ are isomorphic. 
Therefore, if one has a radical solution so does the other.

(Compare with \cite[Lemmas 3.2 and 3.7]{ViDarb}. 
Projective equivalence \refpart{ii} allows the same Darboux covering $\varphi$ trivially. 
It affects the monodromy group by a cyclic direct factor.
Transformation \refpart{iii} offers $1/\varphi$ for the respective Darboux covering.)
\end{proof}

\begin{corollary} \label{th:darb3s}
The same degree $24$ Darboux coverings $\varphi$ or $1/\varphi$ 
for reduction of the projective monodromy $\KG=\rm {PSL}(2,\FF_7)$ 
of hypergeometric equation $(\ref{eq:hpgde})$
to $\ZZ/7\ZZ$ apply to all hypergeometric functions of the types {\rm (3A)} and {\rm (3B)};
or to all  $\hpgo32$-functions of the types {\rm (4A)} and {\rm (4B)}; 
or to all  $\hpgo32$-functions of the types {\rm (7A)} and {\rm (7B)}.
\end{corollary}
\begin{proof}
Each of the six types describes an equivalence class under the contiguity equivalence.
The pairs of types {\rm (3A)}, {\rm (3B)}; or {\rm (4A)}, {\rm (4B)}; or {\rm (7A)}, {\rm (7B)}
are related by transformation \refpart{iv} within the proof of Proposition $\ref{th:3f2cases}$.
\end{proof}

To match hypergeometric functions with radical solutions of a pulled-back Fuchsian equation
for Darboux evaluations, the next 
lemma is useful. It applies to 
equations obtained by the considered  degree 24 pull-back transformations. 
\begin{lemma} \label{th:locexps}
Suppose that differential equation $(\ref{eq:genlde})$ has a finite cyclic monodromy group.
If the local exponents $\lambda_1,\ldots,\lambda_n$ 
at a point $P\in\CC\PP^1$  are all different modulo $\ZZ$,
then for each local exponent $\lambda_i$  ($i\in\{1,\ldots,n\}$) 
there is exactly one (up to scalar multiplication) radical solution 
with the vanishing order $\lambda$ 
at $P$. 
\end{lemma}
\begin{proof} (Compare with \cite[Lemma 3.6]{ViDarb}.)
A monodromy representation of the finite cyclic group is completely reducible. 
Therefore, the solution space is generated by $n$ radical solutions 
(that give the direct decomposition into $n$ one-dimensional spaces).
If two radical solutions had the same local exponent $\lambda^*$ at $P$,
their linear combination would have the vanishing order $\lambda^*+k$ with integer $k>0$.
Hence there is a radical solution for each of the $n$ local exponents.
\end{proof}
A divisor (with coefficients in $\QQ$) of a radical function $f$ on an algebraic curve $B$
is well defined, because an integer power of $f$ is a rational function on $B$.
Conversely, a divisor with $\QQ$-coefficients on $B$ determines 
a radical function up to scalar multiplication.
To compute Darboux evaluations, we usually first determine possible divisors of 
radical solutions (of a pulled-back Fuchsian equation) on a Darboux curve.

\begin{remark} \rm
The simplest Darboux evaluations are reductions of a dihedral monodromy group 
to a cyclic monodromy group by a quadratic transformation \cite{VidunasDh}. 
The Darboux covering is $z=x^2$ in this well-known formula:
\begin{align} \label{eq:dih1}
\hpg21{a,a+\frac{1}2}{\frac12}{\,z} & =\frac{(1-\sqrt{z})^{-2a}+(1+\sqrt{z})^{-2a}}2.
\end{align}
The monodromy group of the hypergeometric equation is finite when $a\in\QQ\setminus\{0\}$. 
Less known are these variations of dihedral formulas:
\begin{align}  \label{eq:dihe1}
\hpg{2}{1}{a,\,-a}{\frac{1}{2}}{\frac{x^2}{x^2-1}} & = \frac12
\left( \left( \frac{1+x}{1-x} \right)^{\!a} + \left( \frac{1-x}{1+x} \right)^{\!a\,} \right),\\  \label{eq:dihe2}
\hpg{2}{1}{\!\frac12+a,\frac12-a}{\frac{1}{2}}{\frac{x^2}{x^2-1}} & = \frac{\sqrt{1-x^2}}2
\left( \left( \frac{1+x}{1-x} \right)^{\!a} + \left( \frac{1-x}{1+x} \right)^{\!a\,} \right),\\  \label{eq:dihe3}
\hpg{2}{1}{\!\frac12+a,\frac12-a}{\frac{3}{2}}{\frac{x^2}{x^2-1}} & = \frac{\sqrt{1-x^2}}{4ax}
\left( \left( \frac{1+x}{1-x} \right)^{\!a} - \left( \frac{1-x}{1+x} \right)^{\!a\,} \right).
\end{align}
\end{remark}

\subsection{Three Darboux coverings}
\label{sec:ourcovs}

By Corollary \ref{th:darb3s}, there are separate sets of Darboux coverings for the type pairs
(3A), (3B), or (4A), (4B), or (7A), (7B). Here we show that these sets consist of single Belyi coverings
(up to holomorphic symmetries of the Darboux curves). The three Belyi maps are
related to each other by algebraic correspondences induced by quadratic and cubic transformations
(\ref{eq:t32a}), (\ref{eq:t32c}) between $\hpgo32$-functions of the different types.
The Darboux curves for the types (4A), (4B) and (7A), (7B) have genus 1.
Their Weierstra\ss{} forms are presented  in formulas (\ref{eq:darbc7}), (\ref{eq:darbc4}).

\begin{propose} \label{th:brpats}
Darboux coverings that reduce the projective monodromy group $\KG=\PSL(2,\FF_7)$
of a third order hypergeometric equation $(\ref{eq:hpgde})$ to $\ZZ/7\ZZ$
are Belyi maps with the following branching patterns (and genus $g$):
\begin{align*}
\mbox{for the types {\rm (3A)} and \rm (3B):} & \qquad  [7^31^3/3^8/2^{12}] & (g=0); \\
\mbox{for the types {\rm (4A)} and \rm (4B):} & \qquad  [7^31^3/4^6/2^{12}] & (g=1); \\
\mbox{for the types {\rm (7A)} and \rm (7B):} & \qquad  [7^31^3/7^31^3/2^{12}] & (g=1).
\end{align*}
\end{propose}
\begin{proof} (Compare with \cite[Lemma 3.4]{ViDarb}.)
Let $E$ denote the third order hypergeometric equation.
Suppose that $f_1,f_2,f_3$ is a basis of its solutions. 
The field extension $\CC(z,f_2/f_1,f_3/f_1)\supset\CC(z)$
defines a Galois covering $\Phi_1$ of $\CC\PP^1$ of degree $\#\KG=168$.
(It is also a Darboux covering that reduces the projective monodromy to a trivial group.)
Its possible branching patterns are $[7^{24}/3^{56}/2^{84}]$, 
$[7^{24}/4^{42}/2^{84}]$ or $[7^{24}/7^{24}/2^{84}]$ (of genus $g=3,10$ or $19$, respectively);
see \cite[\S 8.2]{vdPut00}. 
By the Galois correspondence, a Darboux covering that reduces the monodromy to $\ZZ/7\ZZ$
is a subcovering of $\Phi_1$ of degree $168/7=24$.
The fibers where local exponent differences of $E$ have denominators 
with the least common multiple $k\le 4$ should have only points branching with order $k$. 
The only immediate uncertainty are those fibers with the branching $7^31^3$.
Only branching orders 7, 1 can be present there for subcoverings of $\Phi_1$. 
If there are fewer (than 3) points with the branching order 7, 
the Riemann-Hurwitz formula \cite[p.~404]{Hurwitz93} gives a negative genus for the covering.
\end{proof}
The only Belyi map (up to M\"obius 
transformations on either side of $\PP^1\to\PP^1$)
with the genus $g=0$ branching pattern was computed as such
following \cite[\S 3]{Vidunas2005}, and presented in \cite[p.~178]{ViAGH}.
We refer to this Belyi covering as $\Phi_3$. Its expression is given in (\ref{eq:phi3}).

It is less practical to compute (from scratch) Belyi functions with the two branching patterns of genus $g=1$.
Instead, we follow their algebraic correspondence relations to $\Phi_3$ that are consequences 
of quadratic and cubic transformations 
of $\hpgo32$-functions \cite{KatoTR}.
In particular, representative $\hpgo32$-functions of types {\rm (3A)} and {\rm (7A)} are related 
by cubic transformation (\ref{eq:t32c}). 
For example, consider (\ref{eq:t32c}) with $a=-1/42$, $b=1/14$.
Hence, there exists only one covering $\Phi_7:E_7\to\PP^1$ (up to holomorphic symmetries of both curves)
with the branching pattern $[7^31^3/7^31^3/2^{12}]$.
It can be constructed by considering the fiber product of $\Phi_3$ and the cubic covering $27z^2/(4-z)^3$,
as described in \cite[Lemma 3.5]{ViDarb}. This gives the Darboux curve $E_7$ isomorphic to
\begin{equation} \label{eq:darbc7}
v^2=u\,(1-11u+32u^2), 
\end{equation}
as computed here in \S \ref{sec:dc277}. On the fiber product curve $E_7$, we have
\begin{equation} \label{eq:cub37}
\Phi_3=\frac{27\,\Phi_7^{\,2}}{(4-\Phi_7)^3}.
\end{equation}
Further, representative $\hpgo32$-functions of types {\rm (4B)} and {\rm (7B)} are related 
by quadratic transformation (\ref{eq:t32a}). 
For example, consider (\ref{eq:t32a}) with $a=-1/28$, $b=1/28$. 
Hence, there is one covering $\Phi_4:E_4\to\PP^1$ (up to holomorphic symmetries of both curves)
with the branching pattern $[7^31^3/4^6/2^{12}]$. The fiber product curve is the same $E_7$,
where we have
\begin{equation} \label{eq:qua47}
\Phi_4=-\frac{4\,\Phi_7}{(\Phi_7-1)^2}.
\end{equation}
The degree 24 covering $\Phi_4$ is obtained by dividing out its symmetry $\Phi_7\mapsto 1/\Phi_7$.
Consequently, the Darboux curve $E_4$ is a genus 1 curve that is 
2-isogenous $E_7$; see \S \ref{sec:c247}. 
It is isomorphic to
\begin{equation} \label{eq:darbc4}
w^2=p\,(1+22p-7p^2).
\end{equation}

\section{The cases (3A) and (3B)}
\label{sec:c237}

The $\hpgo32$-functions in (\ref{eq:kleinhpg}) are companion solutions (up to a power factor) 
of the same Fuchsian equation with the projective monodromy $\KG$.
They are of type (3A), and directly parametrize Klein's curve $\KC$.
The degree 24 pull-back covering that reduces the projective monodromy group to $\ZZ/7\ZZ$
can be computed from the degree 168 Galois covering $\Phi_0$ in (\ref{eq:galois168}) and 
the degree 7 projection (\ref{eq:cover7}).
The degree 24 rational function is
\begin{equation} \label{eq:phi3}
\Phi_3=\frac{1728x(x-1)F_1^7}{G_0^3\,G_1^3},
\end{equation}
where
\begin{align}
F_1=& \; 1+5x-8x^2+x^3, \nonumber \\ 
G_0 = & \;1-x+x^2, \\
G_1=& \; 1-235x+1430x^2-1695x^3+270x^4+229x^5+x^6. \nonumber
\end{align}
This is a Belyi function with the branching pattern $[7^31^3/3^8/2^{12}]$.
As recounted in \S \ref{sec:ourcovs}, 
this is a unique Belyi map (up to M\"obius transformations)
with this branching pattern. 
The same Belyi function appears in the hypergeometric formula \cite[(76)]{ViAGH}
\begin{equation}
\hpg21{\frac27,\,\frac37}{\frac67}{x} = G_0^{-1/28}\,G_1^{-1/28}\,
\hpg21{\frac1{84},\,\frac{29}{84}}{\frac67}{\Phi_3}.
\end{equation}
This is not surprising, as branching requirements for pull-back transformations between
hypergeometric equations are similar \cite[\S 3]{ViAGH}.
The Belyi function $\Phi_3$ is presented 
in \cite{Hoshino10} as well.

\subsection{Evaluations of type (3A)}

The following Darboux evaluations are used in \S \ref{sec:level7} to prove 
modular identities (\ref{eq:mod7a})--(\ref{eq:mod7cc}). 
\begin{theorem} \label{eq:th237a}
We have
\begin{align} \label{eq:id3a}
\hpg32{\!-\frac1{42},\,\frac{13}{42},\,\frac{9}{14}}{\frac47,\;\frac67}{\Phi_3}
=&\; (1-x)^{1/7}\,G_0^{-1/14}\,G_1^{-1/14},\\  \label{eq:id3b}
\hpg32{\frac5{42},\,\frac{19}{42},\,\frac{11}{14}}{\frac57,\;\frac87}{\Phi_3}
=&\; (1-x)^{2/7}\,F_1^{-1}\,G_0^{5/14}\,G_1^{5/14},\\  \label{eq:id3c}
\hpg32{\frac{17}{42},\,\frac{31}{42},\,\frac{15}{14}}{\frac97,\;\frac{10}7}{\Phi_3}
=&\; (1-x)^{-3/7}\,F_1^{-3}\,G_0^{17/14}\,G_1^{17/14}
\end{align}
in a neighborhood of $x=0$.
\end{theorem}
\begin{proof}
A hypergeometric differential equation has the following basis of local solutions around $z=0$:
\begin{equation} \label{eq:locs3}  \hspace{-1pt}
\hpg32{\!-\frac1{42},\,\frac{13}{42},\,\frac{9}{14}}{\frac47,\;\frac67}{z}\!,\,
z^{1/7}\hpg32{\!\frac5{42},\,\frac{19}{42},\,\frac{11}{14}}{\frac57,\;\frac87}{z}\!,\, 
z^{3/7}\hpg32{\!\frac{17}{42},\,\frac{31}{42},\,\frac{15}{14}}{\frac97,\;\frac{10}7}{z}\!.
\hspace{-2pt}
\end{equation}
After the pullback $z=\Phi_3(x)$, the local exponents at $x=0$ will be $0,1/7,3/7$
as well. The generalized Riemann's $P$-symbol
\begin{equation}
P\left\{ \begin{array}{cccccc|c}
x=0 & x=1 & x=\infty & F_1(x)\!=\!0 & G_0(x)\!=\!0 & G_1(x)\!=\!0 \\ \hline
0 & 0 & 0 & 0 & -1/14 & -1/14 \\ 
1/7 & 1/7 & 1/7 & 1 & 13/14 & 13/14 & x \\
3/7 & 3/7 & 3/7 & 3 & 27/14 & 27/14 \\
\end{array}  \right\}
\end{equation}
describes all singularities of the pulled-back equation, and the local exponents at them.
The pulled-back equation has a cyclic monodromy,
as we established that (projectively) the function field $\CC(\Phi_3)$ corresponds to $\ZZ/7\ZZ$
by the Galois correspondence in the projective parametrization (\ref{eq:kleinhpg})
of Klein's curve $\KC$ by hypergeometric solutions.
Hence there is a basis of radical solutions.
Candidates for radical solutions are constructed by picking up a local exponent
at each singular point, so that their sum $s$ is either zero or a negative integer. 
If $s<0$, then a presumed solution vanishes (with order 1 or 2)
at some regular points, making the total sum of local exponents at all points equal to zero.
Negative exponents come only from the roots of $G_0G_1$.
We have to pick up the exponent $-1/14$ at those 8 roots,
because a positive exponent in that locus would lead to $s>0$. 
Up to scalar multiplication, the candidate radical solutions for (\ref{eq:id3a}) 
with the local exponent $0$ at $x=0$:
\begin{align}
\psi_1=(1-x)^{1/7}\,G_0^{-1/14}\,G_1^{-1/14}, \qquad
\psi^o_1=(1-x)^{3/7}\,G_0^{-1/14}\,G_1^{-1/14}.
\end{align}
The local exponents at $x=\infty$ equal to $3/7,1/7$, respectively. 
By checking a few initial terms of their 
power series at $x=0$ we see that $\psi_1$ is the actual radical solution.
The candidate solutions with the local exponent $1/7$ at $x=0$ are:
\begin{align}
\psi_2=x^{1/7}(1-x)^{3/7}\,G_0^{-1/14}\,G_1^{-1/14},  \qquad
\psi^o_2=x^{1/7}\,G_0^{-1/14}\,G_1^{-1/14}.
\end{align}
The actual radical solution is $\psi_2$, with the local exponent at $x=\infty$ 
differing from $3/7$ by Lemma \ref{th:locexps}.
This gives formula (\ref{eq:id3b}).
Similarly, the candidate solutions with the local exponent $3/7$ at $x=0$ are:
\begin{align}
\psi_3=x^{3/7}\,G_0^{-1/14}\,G_1^{-1/14}, \qquad
\psi^o_3=x^{3/7}(1-x)^{1/7}\,G_0^{-1/14}\,G_1^{-1/14},
\end{align}
The correct radical solution is $\psi_3$, leading to formula (\ref{eq:id3c}).
\end{proof}

Contiguous relations of $\hpgo32$ functions can be used to derive Darboux evaluations
for other hypergeometric functions in the (3A) class. For example,
applying (\ref{eq:contig11}) to (\ref{eq:id3c}) gives
\begin{equation*}
\hpg32{\!\frac1{14},\,\frac{17}{42},\,\frac{31}{42}}{\frac37,\;\frac97}{\Phi_3\!}
= \frac{(1\!-\!x)^{4/7} G_0^{3/14} G_1^{3/14}}{F_1^{2}} \textstyle
\big(1-\frac{52}9x+\frac{43}3x^2-\frac{16}3x^3+\frac19x^4 \big).
\end{equation*}
This hypergeometric solution vanishes at 4 regular points of its Fuchsian equation,
namely at the roots of the degree 4 polynomial.

The other three companion solutions to (\ref{eq:locs3}) are
\begin{align} \label{eq:locs3i} \hspace{-1pt}
& z^{1/42}\,\hpg32{\!-\frac1{42},\,\frac{5}{42},\,\frac{17}{14}}{\frac13,\;\frac23}{\,\frac1z\,}\!, \quad 
z^{-13/42}\,\hpg32{\frac{13}{42},\,\frac{19}{42},\,\frac{31}{14}}{\frac23,\;\frac43}{\,\frac1z\,}\!, \nonumber \\
& z^{-9/14}\,\hpg32{\frac{9}{14},\,\frac{11}{14},\,\frac{15}{14}}{\frac43,\;\frac53}{\,\frac1z\,}\!.
\hspace{-20pt}
\end{align}
The rational argument in their Darboux evaluations can be $1/\Phi_3$, but the evaluation then holds
in a neighborhood of a root of $G_0G_1=0$. The roots of $G_0$ can be moved to $x=0$, $x=\infty$
by the inverse of the M\"obius transformation
\begin{equation} \label{eq:mobw}
x\mapsto\mu(x)=\frac{x+\omega+1}{\omega(1-x)},
\end{equation}
where $\omega=\exp(2\pi i/3)$.
The point $x=0$ is then a regular point (up to projective equivalence) after a pull-back transformation.
Accordingly, the Darboux evaluations in the next theorem 
express $\hpgo32$-functions as linear combinations of radical solutions. 
They are reminiscent to dihedral evaluations (\ref{eq:dihe1})--(\ref{eq:dihe3}).
\begin{theorem} \label{th:phi33}
The rational function $\Phi_3^*(x)=1/\Phi_3(\mu(x))$ has the expression
\begin{align}
\Phi_3^*= & \; \frac{(24\omega+8)\,x^3\,G_2^3}{(1-x^3)\,F_2^7} 
\end{align}
with
\begin{align}
F_2=& \; 1+\frac{39\omega-16}{49}\,x^3, \\ 
G_2=& \; 1-\frac{435\omega+745}{392}\,x^3+\frac{18357\omega+14632}{16807}\,x^6. \notag
\end{align}
The following formulas hold around $x=0$:
\begin{align} \label{eq:locs1i}
F_2^{1/6}\,
\hpg32{\!-\frac1{42},\,\frac{5}{42},\,\frac{17}{42}}{\frac13,\;\frac23}{\Phi^*_3} 
\! = \hspace{-42pt} \\
& \, \frac13\,(1-x)^{-1/42}\,(1-\omega x)^{5/42}\,(1-\omega^2x)^{17/42} \nonumber \\
 & + \frac13\,(1-x)^{5/42}(1-\omega x)^{17/42}(1-\omega^2x)^{-1/42} \nonumber \\
& + \frac13\,(1-x)^{17/42}(1-\omega x)^{-1/42}(1-\omega^2x)^{5/42},
 \nonumber  \displaybreak  \\[1pt]
 \frac{x\,G_2}{1-2\omega}\,F_2^{-13/6}\,
\hpg32{\frac{13}{42},\,\frac{19}{42},\,\frac{31}{42}}{\frac23,\;\frac43}{\Phi^*_3}
\! = \hspace{-42pt} \\
& \,  \frac13\,(1-x)^{13/42}\,(1-\omega x)^{19/42}\,(1-\omega^2x)^{31/42} \nonumber \\
 & +\frac{\omega}3\,(1-x)^{19/42}(1-\omega x)^{31/42}(1-\omega^2x)^{13/42} \nonumber \\
& +\frac{\omega^2}3\,(1-x)^{31/42}(1-\omega x)^{13/42}(1-\omega^2x)^{19/42},
 \nonumber\\[1pt]
 \frac{3x^2\,G_2^2}{21+7\omega}\,F_2^{-9/2}\,
\hpg32{\frac{9}{14},\,\frac{11}{14},\,\frac{15}{14}}{\frac43,\;\frac53}{\Phi^*_3}
\! = \hspace{-42pt} \\
& \,  \frac13\,(1-x)^{9/14}\,(1-\omega x)^{11/14}\,(1-\omega^2x)^{15/14} \nonumber \\
 & +\frac{\omega^2}3\,(1-x)^{11/14}(1-\omega x)^{15/14}(1-\omega^2x)^{9/14} \nonumber \\
& +\frac{\omega}3\,(1-x)^{15/14}(1-\omega x)^{9/14}(1-\omega^2x)^{11/14}. \nonumber
\end{align}
\end{theorem}
\begin{proof}
Up to constant multiples, 
substitution (\ref{eq:mobw}) transforms the solutions $\psi_1$, $\psi_2$, $\psi_3$ 
of the previous proof to
\begin{align}
\psi^*_1 & =(1-x)^{3/7}\,(1-\omega^2x)^{1/7}\,x^{-1/14}\,G_2^{-1/14}, \notag \\
\psi^*_2 & =(1-\omega x)^{1/7}\,(1-\omega^2x)^{3/7}\,x^{-1/14}\,G_2^{-1/14},\\
\psi^*_3 & =(1-x)^{1/7}\,(1-\omega x)^{3/7}\,x^{-1/14}\,G_2^{-1/14}. \notag
\end{align}
By considering the first few terms of Poisson power series in $x$,
we express each local solution in (\ref{eq:locs3i}) with $z=1/\Phi^*_3$ as linear combinations
of $\psi^*_1$, $\psi^*_2$, $\psi^*_3$. The claimed hypergeometric identities are then obtained
after simplifying the powers of $x,G_2$ and bringing the linear factors 
of $1-x^3=(1-x)(1-\omega x)(1-\omega^2x)$ to the right-hand side.
\end{proof}
Evidently, $\Phi^*_3(x)$ is a compositions of a degree 8 covering with the cyclic $x\mapsto x^3$ covering.
The M\"obius-equivalent function $\Phi_3(x)$ is a composition of degree 8 and 3 coverings as well.
This is noted in 
\cite[\S 9]{ViAGH}, \cite[\S 4]{Hoshino10}.

\subsection{Evaluations of type (3B)}

The simplest representative evaluations of this type involve polynomial
parts vanishing at some regular points of the hypergeometric equation.
\begin{theorem} \label{eq:th237b}
We have
\begin{align} \label{eq:locs1b}
\hpg32{\!-\frac1{14},\,\frac{11}{42},\,\frac{25}{42}}{\frac47,\;\frac57}{\Phi_3}
=&\; (1-3x)\,(1-x)^{3/7}\,G_0^{-3/14}\,G_1^{-3/14},\\
\hpg32{\frac3{14},\,\frac{23}{42},\,\frac{37}{42}}{\frac67,\;\frac97}{\Phi_3}
=&\; \left(1-\frac{2x}3\right) (1-x)^{-2/7}\,F_1^{-2}\,G_0^{9/14}\,G_1^{9/14},\\
\hpg32{\frac{5}{14},\,\frac{29}{42},\,\frac{43}{42}}{\frac87,\;\frac{10}7}{\Phi_3}
=&\; \left(1+\frac{x}2\right) (1-x)^{-1/7}\,F_1^{-3}\,G_0^{15/14}\,G_1^{15/14}
\end{align}
in a neighborhood of $x=0$. 
\end{theorem}
\begin{proof}
A hypergeometric differential equation has the following basis of local solutions around $z=0$:
\begin{equation} \label{eq:locs3b} \hspace{-2pt}
\hpg32{\!-\frac1{14},\,\frac{11}{42},\,\frac{25}{42}}{\frac47,\;\frac57}{z}\!,\, 
z^{2/7}\hpg32{\!\frac3{14},\,\frac{23}{42},\,\frac{37}{42}}{\frac67,\;\frac97}{z}\!,\,
z^{3/7}\hpg32{\!\frac{5}{14},\,\frac{29}{42},\,\frac{43}{42}}{\frac87,\;\frac{10}7}{z}\!.
\hspace{-3pt}
\end{equation}
The pullback with respect to $z=\Phi_3(x)$ gives a Fuchsian equation with 
the generalized Riemann's $P$-symbol
\begin{equation}
P\left\{ \begin{array}{cccccc|c}
x=0 & x=1 & x=\infty & F_1(x)\!=\!0 & G_0(x)\!=\!0 & G_1(x)\!=\!0 \\ \hline
0 & 0 & 0 & 0 & -3/14 & -3/14 \\ 
2/7 & 2/7 & 2/7 & 2 & 11/14 & 11/14 & x \\
3/7 & 3/7 & 3/7 & 3 & 25/14 & 25/14 \\
\end{array}  \right\}.
\end{equation}
Candidate radical solutions with the local exponent 0 at $x=0$ are:
\begin{align}
\psi_4 & =(1-c_1x)\,(1-x)^{2/7}\,G_0^{-3/14}\,G_1^{-3/14}, \\
\psi^\star_4 & =(1-c_2x)\,(1-x)^{3/7}\,G_0^{-3/14}\,G_1^{-3/14}, \notag
\end{align}
where the coefficients $c_1,c_2\in\CC\setminus\{0,1\}$ are undetermined yet.
A priori, $c_1$ or $c_2$ could be a root of $G_0G_1=0$, 
so that the local exponent at that root equals $11/14$.
The correct solution is $\psi^\star_4$ with $c=3$, leading to (\ref{eq:locs1b}).
The other two identities are proved similarly.
\end{proof}

Application of M\"obius transformation (\ref{eq:mobw})
gives identities for
\begin{align}
& \sqrt{F_2}\;
\hpg32{\!-\frac1{14},\,\frac{3}{14},\,\frac{5}{14}}{\frac13,\;\frac23}{\Phi^*_3}, \quad
\frac{6x\,G_2}7\,F_2^{-11/6}\,
\hpg32{\frac{11}{42},\,\frac{23}{42},\,\frac{29}{42}}{\frac23,\;\frac43}{\Phi^*_3},\quad\nonumber\\
& \frac{3x^2G_2^2}{14\omega\!-\!7}\,
F_2^{-25/6} \; \hpg32{\frac{25}{42},\,\frac{37}{42},\,\frac{43}{42}}{\frac43,\;\frac53}{\Phi^*_3},
\end{align}
similar to those in Theorem \ref{th:phi33}.

\section{Modular relations}
\label{sec:modular}

The formulas of Theorem \ref{eq:th237a} can be used
to prove the modular product identities for 
$K_1(\tau),K_2(\tau),K_3(\tau)$ in (\ref{eq:mod7a})--(\ref{eq:mod7cc}).
This is demonstrated in  \S \ref{sec:level7}. Recalling Dedekind's eta-function \cite[p.~29]{Zagier08}
\begin{equation} \label{eq:defeta}
\eta(\tau) = q^{1/24}\,\prod_{n=1}^{\infty} (1-q^n)=\sum_{n=-\infty}^{\infty} q^{\frac1{24}{(6n+1)^2}},
\end{equation}
the functions $\eta(\tau)^4\,K_1(\tau)$, $\eta(\tau)^4\,K_2(\tau)$, $\eta(\tau)^4\,K_3(\tau)$
are known to be modular forms of weight 2 on the modular curve $\XX(7)$;
see \cite[Example 29]{FrancMason}. 
Therefore we refer to the formulas in  (\ref{eq:mod7a})--(\ref{eq:mod7cc})
as {\em level $7$ evaluations}.

As a warm-up, we prove in \S \ref{sec:rrq} hypergeometric expressions (\ref{eq:rr1a}), (\ref{eq:rr2a})
for the (slightly modified) Rogers-Ramanujan series. 
Because of ubiquitous relation \cite{Duke} to the modular curve $\XX(5)$,
we refer to those formulas as {\em level $5$ evaluations}.
Multiplied by $\eta(\tau)^{2/5}$, the functions in  (\ref{eq:rr1a})--(\ref{eq:rr2b})
are considered in \cite{Kaneko06} as modular forms of weight $1/5$ on $\XX(5)$.
In \S \ref{sec:level234} we show a few more similar modular evaluations of $\hpgo21$-functions.

For an introduction to modular curves and functions, we refer to \cite{DiamondS}. 
For a positive integer $N$, 
the modular curves $\XX(N)$, $\XX_1(N)$, $\XX_0(N)$ are defined by the congruence subgroups
$\Gamma(N)$, $\Gamma_1(N)$, $\Gamma_0(N)$, respectively, of $\mbox{\rm SL}(2,\ZZ)$.
If a moduli curve has genus 0, then a generator of its field of moduli functions
is called a {\em Hauptmodul}. For example, the $j$-invariant (\ref{eq:jinv}) is a Hauptmodul of $\XX(1)$,
and
\begin{align} 
\label{eq:h2} h_2(\tau):= & \, \frac{\eta(\tau)^{24}}{\eta(2\tau)^{24}}
= \frac1{q}\,\prod_{n=1}^{\infty} \big(1+q^n\big)^{-24}, \\
\label{eq:h3} h_3(\tau):= & \, \frac{\eta(\tau)^{12}}{\eta(3\tau)^{12}}
=\frac1{q}\,\prod_{n=0}^{\infty} \big(1-q^{3n+1}\big)^{\!12} \, \big(1-q^{3n+2}\big)^{\!12}, \\
\label{eq:h4} h_4(\tau):= & \; \frac{\eta(\tau)^{8}}{\eta(4\tau)^{8}} \;
= \frac1{q}\,\prod_{n=1}^{\infty} \big(1+q^{2n-1}\big)^{-8} \,\big(1+q^{2n}\big)^{-16}, \\
\label{eq:h5} h_5(\tau) := & \; \frac{\eta(\tau)^6}{\eta(5\tau)^6} \;
=\frac1q\,\prod_{n=1}^{\infty} \frac{(1-q^n)^6}{(1-q^{5n})^6}, \\
\label{eq:h7} h_7(\tau) := & \; \frac{\eta(\tau)^4}{\eta(7\tau)^4} \;
=\frac1q\,\prod_{n=1}^{\infty} \frac{(1-q^n)^4}{(1-q^{7n})^4}
\end{align}
are Hauptmoduln of $\XX_0(N)$ for $N=2,3,4,5,7$, respectively \cite[Table 1]{Maier07}.
The genus of modular curves and branching patterns of maps between them 
can be easily derived the tables of Cummins and Pauli \cite{CumminsP}.

Existence of Fuchsian equations for modular forms with differentiation with respect to a modular function
is well-known \cite[\S 5.4]{Zagier08}. They are commonly known as {\em Picard-Fuchs equations}. 

\subsection{A proof of the level 5 evaluations} 
\label{sec:rrq}

This proof of formulas (\ref{eq:rr1a}), (\ref{eq:rr2a}) 
follows the argument in \cite{stackex}, though we seek to prepare a contextual template
for the proof in \S \ref{sec:level7}. We start with presenting Darboux evaluations \cite[\S 2.3]{ViDarb} 
for the involved hypergeometric functions .

The hypergeometric functions in (\ref{eq:rr1a}), (\ref{eq:rr2a}) 
are standard $\hpgo21$-functions with 
the icosahedral projective monodromy group \cite{ViDarb}.
The projective monodromy group can be reduced to $\ZZ/5\ZZ$ 
by a pull-back transformation with respect to this degree 12 covering:
\begin{equation} \label{eq:pback12}
\varphi_5(x)=\frac{1728\,x\,(1-11x-x^2)^5}{(1+228x+494x^2-228x^3+x^4)^3}.
\end{equation}
This differs from \cite[(2.9)]{ViDarb} by the transformations $x\mapsto -x$ or $x\mapsto 1/x$. 
We rewrite the Darboux evaluations in \cite[(2.9)--(2.10)]{ViDarb} as follows:
\begin{align} \label{eq:icosa1}
\left(\frac{\varphi_5(x)}{1728}\right)^{\!-1/60}
\hpg21{\!-\frac1{60},\,\frac{19}{60}}{\frac45}{\varphi_5(x)} & = x^{-1/60}\,(1-11x-x^2)^{-1/12},\\
 \label{eq:icosa2} \left(\frac{\Phi_5(x)}{1728}\right)^{11/60}
\hpg21{\frac{11}{60},\,\frac{31}{60}}{\frac65}{\varphi_5(x)} & = x^{11/60}\,(1-11x-x^2)^{-1/12}.
\end{align}
%
The modular curves $\XX(5)$, $\XX_1(5)$, $\XX_0(5)$, $\XX(1)$ have genus $0$, as is well-known.
The maps between them have the degrees indicated in this diagram:
\begin{equation}
\XX(5)\stackrel{5}{\longrightarrow}\XX_1(5)\stackrel{2}{\longrightarrow}
\XX_0(5)\stackrel{6}{\longrightarrow}\XX(1).
\end{equation} 
A Hauptmodul of $\XX_1(5)$ is 
\begin{equation} \label{eq:x5}
x_5(\tau) = q\; \prod_{n=0}^{\infty} \frac{(1-q^{5n+1})^5\,(1-q^{5n+4})^5}{(1-q^{5n+2})^5\,(1-q^{5n+3})^5}.
\end{equation}
A Hauptmodul of $\XX(5)$ is $y_5(\tau)=x_5(\tau)^{1/5}$. 
It has this nice expression as a continuous fraction due to Ramanujan \cite{Duke}: 
\begin{align*}
y_5(\tau) 
= \frac{q^{1/5}}{\displaystyle 1+\frac{q}{\displaystyle 1+\frac{q^2}{\displaystyle 1+\frac{q^3}{1+\ldots}}}}.
\end{align*}  
A Hauptmodul of $\XX_0(5)$ is given in (\ref{eq:h5}). As a function on $\XX_1(5)$, 
it is identified as
\begin{equation}
h_5(\tau)=\frac1{x_5(\tau)}-11-x_5(\tau).
\end{equation}
The covering $\XX_1(5)\rightarrow\XX(1)$ is
\begin{equation}
j(\tau)=\frac{1728}{\varphi_5(x_5(\tau))},
\end{equation}
and the covering $\XX(5)\rightarrow\XX(1)$ is Klein's icosahedral Galois covering 
\begin{equation}
j(\tau)=\frac{1728}{\varphi_5(y_5(\tau)^5)}.
\end{equation}
The identification $x=x_5(\tau)$ in (\ref{eq:pback12})
establishes 
\begin{equation}
\varphi_5(x)=\frac{1728}{j(\tau)}, \qquad
1-11x-x^2=x_5(\tau)h_5(\tau).
\end{equation}
This evaluates the right-hand sides of (\ref{eq:icosa1})--(\ref{eq:icosa2}) to, respectively,
\begin{equation}
x_5(\tau)^{-1/10}\,h_5(\tau)^{-1/12} \qquad\mbox{and}\qquad
x_5(\tau)^{1/10}\,h_5(\tau)^{-1/12}.
\end{equation}
Formulas  (\ref{eq:rr1a}), (\ref{eq:rr2a}) follow from (\ref{eq:h5}) and  (\ref{eq:x5}) now.

\begin{remark} \rm \label{rm:level5}
The Darboux evaluations  (\ref{eq:icosa1})--(\ref{eq:icosa2}) with $x=y_5(\tau)^5$ imply that
the two functions are defined on the Galois covering 
\begin{equation} \label{eq:rrdefc}
z^{12}=y\,(1-11y^5-y^{10})
\end{equation}
of $\XX(5)$. 
This is the minimal curve where the modular functions in (\ref{eq:rr1a}), (\ref{eq:rr2a}) are defined.
The genus of the covering equals $55=1+\frac12(12\,(0-2)+12\,(12-1))$ by the Riemann-Hurwitz formula. 
The composition with $\XX(5)\to\XX(1)$ gives a Belyi covering with the branching pattern
$[60^{12}/3^{240}/2^{360}]$. 
The covering $\XX(60)\to\XX(1)$ has a similar branching pattern but of degree 69120;
see \cite[p.~101]{DiamondS}.
The modular curve $\XX(60)$ ought to be an unramified covering of (\ref{eq:rrdefc})
of degree $96=69120/720$.
\end{remark}

\subsection{A proof of the level 7 evaluations} 
\label{sec:level7}

This proof of formulas (\ref{eq:mod7a})--(\ref{eq:mod7cc}) 
follows the same pattern as in \S \ref{sec:rrq}.
We observe that Hauptmoduln or generating functions of relevant modular curves
have nice $q$-factorizations. The Darboux evaluations of Theorem \ref{eq:th237a}
can be rewritten in a way where each factor can be recognized in terms of those nice 
modular functions. In particular, we rewrite 
(\ref{eq:id3a})--(\ref{eq:id3c}) as follows:
\begin{align} \label{eq:dklein1}
\left(\frac{\Phi_3}{1728}\right)^{\!-1/42}
\hpg32{\!-\frac1{42},\,\frac{13}{42},\,\frac{9}{14}}{\frac47,\;\frac67}{\Phi_3}
& = (-x)^{-1/42}\,(1-x)^{5/42}\,F_1^{-1/6},\\
\label{eq:dklein2} \left(\frac{\Phi_3}{1728}\right)^{5/42}\;
\hpg32{\frac5{42},\,\frac{19}{42},\,\frac{11}{14}}{\frac57,\;\frac87}{\Phi_3}
& = (-x)^{5/42}\,(1-x)^{17/42}\,F_1^{-1/6},\\
\label{eq:dklein3} \left(\frac{\Phi_3}{1728}\right)^{17/42}\,
\hpg32{\frac{17}{42},\,\frac{31}{42},\,\frac{15}{14}}{\frac97,\;\frac{10}7}{\Phi_3}
& = (-x)^{17/42}\,(1-x)^{-1/42}\,F_1^{-1/6}.
\end{align}
Recall that $F_1=x^3-8x^2+5x+1$.

The following coverings hold between modular curves, of indicated degrees:
\begin{equation}
\XX(7)\stackrel{7}{\longrightarrow}\XX_1(7)\stackrel{3}{\longrightarrow}
\XX_0(7)\stackrel{8}{\longrightarrow}\XX(1).
\end{equation} 
The modular curve $\XX(7)$ has the genus $g=3$, 
and is isomorphic to Klein's quartic curve $\KC$ defined by (\ref{eq:klein}).
Elkies \cite[(4.4)--(4.6)]{Elkies} gives this parametrization of $\KC$ 
by modular forms\footnote{These are essentially the modular forms
mentioned right after formula (\ref{eq:defeta}); see \cite[Example 29]{FrancMason}. 
Specifically, $X=-\eta(\tau)^4\,K_3(\tau)$, $Y=\eta(\tau)^4\,K_2(\tau)$, $Z=\eta(\tau)^4\,K_1(\tau)$.
} 
on $\XX(7)$:
\begin{align} 
X=& -q^{4/7}\,\prod_{n=1}^{\infty} (1-q^n)^3\,(1-q^{7n})(1-q^{7n-6})(1-q^{7n-1}), \quad\nonumber \\
\label{eq:elk7b} 
Y=& \; q^{2/7}\,\prod_{n=1}^{\infty} (1-q^n)^3\,(1-q^{7n})(1-q^{7n-5})(1-q^{7n-2}),\\
Z=& \; q^{1/7}\,\prod_{n=1}^{\infty} (1-q^n)^3\,(1-q^{7n})(1-q^{7n-4})(1-q^{7n-3}). \nonumber
\end{align}
The curves $\XX_1(7)$ and $\XX_0(7)$ have genus 0.
A Hauptmodul for $\XX_1(7)$ is 
\begin{align} \label{eq:defx7}
x_7(\tau)= & -\frac{X^2 Y}{Z^3} \\
= & -q+2q^2-5q^4+4q^5+O(q^6), \nonumber
\end{align}
consistent with (\ref{eq:cover7}). 
A Hauptmodul for $\XX_0(7)$ is given in (\ref{eq:h7}).
As a function on $\XX_1(7)$,  it is identified in \cite[(4.24)]{Elkies}:
\begin{equation}
h_7(\tau)=\frac{x_7^{\,3}-8x_7^{\,2}+5x_7+1}{x_7^{\,2}-x_7}.
\end{equation}
The covering $\XX_0(7)\rightarrow\XX(1)$ is given in \cite[(4.20)]{Elkies}:
\begin{equation}
j(\tau)=\frac{(h_7^{\,2}+13h_7+49)(h_7^{\,2}+245h_7+7^4)^3}{h_7^{\,7}}.
\end{equation}
Consequently, the covering $\XX_1(7)\rightarrow\XX(1)$ 
is given by 
\begin{equation}
j(\tau)=\frac{1728}{\Phi_3(x_7)}.
\end{equation}
The identification $x=x_7(\tau)$ 
establishes 
\begin{equation}
\Phi_3(x)=\frac{1728}{j(\tau)}, \qquad
F_1=(-x)(1-x)h_7(\tau),
\end{equation}
and
\begin{equation}
-x=\frac{(-X)^2\,Y}{Z^3},\qquad 1-x=\frac{Y^3}{(-X)\,Z^2}.
\end{equation}
This evaluates the right-hand sides of (\ref{eq:dklein1})--(\ref{eq:dklein3}) to, respectively,
\begin{align} \label{eq:devxyz}
(-X)^{-1/3}\,Y^{-1/3}\,Z^{2/3}\,h_7(\tau)^{-1/6}, \nonumber \\
(-X)^{-1/3}\,Y^{2/3}\,Z^{-1/3}\,h_7(\tau)^{-1/6}, \\
(-X)^{2/3}\,Y^{-1/3}\,Z^{-1/3}\,h_7(\tau)^{-1/6}. \nonumber
\end{align}
Formulas (\ref{eq:mod7aa}), (\ref{eq:mod7bb}), (\ref{eq:mod7cc}) for, respectively,
$K_1(\tau),K_2(\tau),K_3(\tau)$ follow from 
the parametrization (\ref{eq:elk7b}) now.

\begin{remark} \rm
Corrected\footnote{There are these typos in \cite{Elkies}:
Formula (4.19) is actually for $j_7^{-1}$, not for $j_7$.
In formula (4.23), a minus sign is missing before ${\sf y}^2{\sf z}/{\sf x}^3$.
Formula (1.18) should be adjusted by \mbox{$40\Phi_4^2\Phi_6
(3\Phi_{14}^2+1008\Phi_4\Phi_6^4+56\Phi_4^2\Phi_6\Phi_{14}-832\Phi_4^4\Phi_6^2-256\Phi_4^7)$.}
In formula (4.35), the factor $5\phi^2-15\phi-7$ should be $5\phi^2-14\phi-7$.
}
formula \cite[(4.19)]{Elkies} states
\begin{equation}
h_7(\tau)=\frac{R_6(X,Y,Z)}{X^2Y^2Z^2},
\end{equation}
where $R_6(X,Y,Z)$ is the invariant in (\ref{eq:inv6}).
Hence the functions in (\ref{eq:devxyz}) are equal to
$ZR_6^{-1/6}$, $YR_6^{-1/6}$, $XR_6^{-1/6}$, respectively. 
It follows that the functions $K_1(\tau),K_2(\tau),K_3(\tau)$ 
are defined on the Galois covering 
\begin{equation} \label{eq:rrdefd}
W^6=XY^5+YZ^5+ZX^5-5X^2Y^2Z^2
\end{equation}
of $\XX(7)$. There are 24 branching points on $W=0$,
thus the genus of the covering equals $73=1+\frac12(6\,(2\cdot3-2)+24\,(6-1))$ by the Riemann-Hurwitz formula.
The composition with $\XX(7)\to\XX(1)$ gives a Belyi covering with the branching pattern
$[42^{24}/3^{336}/2^{504}]$. 
The covering $\XX(42)\to\XX(1)$ has a similar branching pattern but of degree 24192;
see \cite[p.~101]{DiamondS}.
The modular curve $\XX(42)$ ought to be an unramified covering of (\ref{eq:rrdefd})
of degree $24=24192/1008$.
\end{remark}

\begin{remark} \rm
One may naturally ask about expressing $K_1(\tau),K_2(\tau),K_3(\tau)$ 
as sums similar to the Rogers-Ramanujan summations in  (\ref{eq:rr1b}), (\ref{eq:rr2b}).  
Duke \cite{Duke} refers to the work  \cite{Selberg36} of Selberg
and indicates these expressions:
\begin{align}
K_1(\tau)= &\, \frac{q^{-1/42}}{(1-q)(1-q^2)\cdots\;}\sum_{n=0}^{\infty} (-1)^n\,q^{\frac{7n^2+n}2}\,(1-q^{6n+3}), \nonumber\\
K_3(\tau)= &\, \frac{q^{17/42}}{(1-q)(1-q^2)\cdots\;}\sum_{n=0}^{\infty} (-1)^n\,q^{\frac{7n^2+7n}2}
\,(1-q^{n+1})\,(1-q^{6n+6}),\\
K_2(\tau)= &\; q^{1/7}F_1+
\frac{q^{5/42}\,q^2}{(1-q)(1-q^2)\cdots\;}\sum_{n=0}^{\infty} (-1)^n\,q^{\frac{7n^2+13n}2}
\,C_n\,(1-q^{6n+9}), \qquad\nonumber
\end{align}
where $C_n=(1-q^{n+1})(1-q^{n+2})$. 
If Klein's parametrization \cite[(44)]{Klein79} of $\KC$ is adjusted by $q\mapsto\sqrt{q}$
(and some summation shifts), it gives these beautiful identities:
\begin{align}
\frac{K_3(\tau)}{K_2(\tau)}=& \frac{q^{2/7}\,\sum_{-\infty}^{\infty} (-1)^n q^{\frac{21n^2+7n}2}}
{\sum_{-\infty}^{\infty} (-1)^n q^{\frac{21n^2+n}2}+\sum_{-\infty}^{\infty} (-1)^n q^{\frac{21n^2+13n+2}2}},\\
\frac{K_2(\tau)}{K_1(\tau)}=& \frac{q^{1/7}\,\sum_{-\infty}^{\infty} (-1)^n q^{\frac{21n^2+7n}2}}
{\sum_{-\infty}^{\infty} (-1)^n q^{\frac{21n^2-5n}2}+\sum_{-\infty}^{\infty} (-1)^n q^{\frac{21n^2+19n+4}2}},\\
\frac{K_1(\tau)}{K_3(\tau)}=& \frac{q^{-3/7}\,\sum_{-\infty}^{\infty} (-1)^n q^{\frac{21n^2+7n}2}}
{\sum_{-\infty}^{\infty} (-1)^n q^{\frac{21n^2-11n}2}+\sum_{-\infty}^{\infty} (-1)^n q^{\frac{21n^2+25n+6}2}}.
\end{align}
The numerator sums are equal to $\prod_{n=1}^{\infty}(1-q^{7n})$ by (\ref{eq:defeta}).
Consequently, the denominators can be expressed as nice $q$-products as well.
Those expressions are special cases  (with $s=q^7$, $y\in\{q^{-1},q^{-2},q^{-3}\}$) of this version
of the quintuple product identity: 
\begin{align}
 & \prod_{n=1}^{\infty} \frac{(1-s^n)(1-y^2s^n)(1-y^{-2}s^{n-1})}{(1-ys^n)\,(1-y^{-1}s^{n-1})}= \\
 &  \hspace{110pt} \sum_{n=-\infty}^{\infty} \! (-1)^n(y^{3n-1}+y^{-3n})\,s^{\frac{3n^2-n}{2}}. \qquad\nonumber
\end{align}
Compared with Watson's formula \cite[(1.3)]{Watson5}, we have $x=\sqrt{s}$, $a=-y\sqrt{s}$,
and a partial change of the summation index.
\end{remark}

\subsection{Other similar evaluations}
\label{sec:level234}

Here we relate hypergeometric functions with a dihedral, tetrahedral, or octahedral 
projective monodromy groups with modular curves of level $N=2$, 3 or 4.
Correspondingly, the following full coverings $\XX(N)\to\XX(1)$ have exactly those monodromy groups:
\begin{align}
\mbox{a dihedral } S_3: & \qquad 
\XX(2)\stackrel{2}{\longrightarrow}\XX_1(2)\cong\XX_0(2)\stackrel{3}{\longrightarrow}\XX(1); \nonumber \\
\mbox{the tetrahedral }  A_4: & \qquad 
\XX(3)\stackrel{3}{\longrightarrow}\XX_1(3)\cong\XX_0(3)\stackrel{4}{\longrightarrow}\XX(1); \\
\mbox{the octahedral }  S_4: & \qquad 
\XX(4)\stackrel{4}{\longrightarrow}\XX_1(4)\cong\XX_0(4)\stackrel{6}{\longrightarrow}\XX(1). \nonumber
\end{align}
To obtain Darboux evaluations \cite{ViDarb} of standard Gauss hypergeometric functions
with the tetrahedral or octahedral projective monodromies,
the following pull-back transformations reduce the projective monodromy
to $\ZZ/3\ZZ$ or $\ZZ/4\ZZ$, respectively: 
\begin{equation}
\varphi_3(x)=\frac{x\,(x+4)^3}{4(2x-1)^3}, \qquad
\varphi_4(x)=\frac{108x\,(x-1)^4}{(x^2+14x+1)^3}.
\end{equation}
As we will see, these coverings are algebraically equivalent to the modular coverings 
\mbox{$\XX_0(3)\rightarrow\XX(1)$}, \mbox{$\XX_0(4)\rightarrow\XX(1)$}, respectively.
We will use the Darboux evaluations \cite[(2.1)--(2.2),(2.5)--(2.6)]{ViDarb} in this form:
\begin{align} \label{eq:tetr1}
\left(\frac{\varphi_3(x)}{1728}\right)^{\!-1/12}
\hpg21{\!-\frac1{12},\,\frac{1}{4}}{\frac23}{\varphi_3(x)} & = 
\left(-\frac{x}{108}\right)^{-1/12}\,\left(1+\frac{x}4\right)^{-1/4},\\
 \label{eq:tetr2} \left(\frac{\varphi_3(x)}{1728}\right)^{\!1/4}\,
\hpg21{\frac{1}{4},\,\frac{7}{12}}{\frac43}{\varphi_3(x)} & = 
\left(-\frac{x}{108}\right)^{1/4}\,\left(1+\frac{x}4\right)^{-1/4},\\
\label{eq:octa1} \left(\frac{\varphi_4(x)}{1728}\right)^{\!-1/24}
\hpg21{\!-\frac1{24},\,\frac{7}{24}}{\frac34}{\varphi_4(x)} & = 
\left(\frac{x}{16}\right)^{-1/24}\,(1-x)^{-1/6},\\
 \label{eq:octa2} \left(\frac{\varphi_4(x)}{1728}\right)^{5/24}
\hpg21{\frac{5}{24},\,\frac{13}{24}}{\frac54}{\varphi_4(x)} & = 
\left(\frac{x}{16}\right)^{5/24}\,(1-x)^{-1/6}.
\end{align}
Formula (\ref{eq:tetr2}) is equivalent to (\ref{eq:tetra2}), in particular.

We can reduce the dihedral projective monodromy group to $\ZZ/2\ZZ$ 
using identities like (\ref{eq:dihe1})--(\ref{eq:dihe3}),
but that is not consistent with the cubic modular covering $\XX_0(2)\rightarrow\XX(1)$.
Instead, we apply standard cubic transformation \cite[(21)]{ViAGH} of $\hpgo21$-functions
(with $a\in\{1/2,-1\}$) and get
\begin{align} \label{eq:dihb1}
\left(\frac{\varphi_2(x)}{1728}\right)^{\!-1/6}
\hpg21{\!-\frac1{6},\,\frac{1}{6}}{\frac12}{\varphi_2(x)} & = 
\left(\frac{x}{64}\right)^{-1/6}\,\left(1-x\right)^{-1/3},\\   \label{eq:dihb2}
 \left(\frac{\varphi_2(x)}{1728}\right)^{\!1/3}\,
\hpg21{\frac{1}{3},\,\frac23}{\frac32}{\varphi_2(x)} & = 
\left(\frac{x}{64}\right)^{1/3}\,\left(1-x\right)^{-1/3},
\end{align}
with
\begin{equation}
\varphi_2(x)=\frac{27x(1-x)^2}{(1+3x)^3}.
\end{equation}
In terms of Hauptmoduln, the map $\XX_0(2)\rightarrow\XX(1)$ is given by
\begin{equation}
j(\tau)=\frac{(h_2(\tau)+256)^3}{h_2(\tau)^2}.
\end{equation}
In \cite[Appendix]{Maier07} for a reference, $x_2=2^{12}/h_2(\tau)$ and $j(\tau)=(x_2+16)^3/x_2$.
To match $1728/j(\tau)$ with $\varphi_2(x)$, we identify
\begin{equation}
x=\frac{64}{h_2(\tau)+64}.
\end{equation}
Then $1-x=h_2(\tau)/(h_2(\tau)+64)$, and formulas (\ref{eq:dihb1})--(\ref{eq:dihb2}) become
\begin{align} \label{eq:dihj1}
j(\tau)^{1/6}\;\hpg21{\!-\frac{1}{6},\frac1{6}}{\frac12}{\frac{1728}{j(\tau)}} 
& = h_2(\tau)^{-1/3}\;\sqrt{h_2(\tau)+64}, \\
j(\tau)^{-1/3}\,\hpg21{\frac1{3},\,\frac{2}{3}}{\frac32}{\frac{1728}{j(\tau)}} 
& = h_2(\tau)^{-1/3}.
\end{align}
The function $h_2(\tau)^{-1/3}$ is invariant under the congruence subgroup 6D$^0$
in the tables of Cummins and Paule \cite{CumminsP},
as it gives the branching pattern \mbox{$[6\!\cdot\!3/3^3/2^31^3]$} over the $j$-line. 
The quotient $\sqrt{h_2(\tau)+64}$ of the two functions is a Hauptmodul of $\XX(2)$.
Apparently it has no nice $q$-factorization. 
The standard Hauptmodul of $\XX(2)$ is 
Legendre's modular function \cite[p.~63]{Zagier08}
\begin{equation}
\lambda(\tau)=\frac{16\,\eta(\tau/2)^8\,\eta(2\tau)^{16}}{\eta(\tau)^{24}} =
16\sqrt{q}\,\prod_{n=0}^{\infty} \frac{\big(1+q^n\big)^8}{\big(1+\sqrt{q}\,q^{n}\big)^8}.
\end{equation}
By solving algebraic relations and checking the series, we indeed have
\begin{equation}
h_2(\tau)=256\,\frac{1-\lambda(\tau)}{\lambda(\tau)^2}, \qquad 
\sqrt{h_2(\tau)+64}=8\left(\frac{2}{\lambda(\tau)}-1\right).
\end{equation}
The function in (\ref{eq:dihj1}) is invariant under the subgroup 6C$^1$ in \cite{CumminsP},
as it contains both $\Gamma(2)$ and 6D$^0$.
It has the branching pattern \mbox{$[6^3/3^6/2^9]$}.
The covering over $\XX(2)$ can be obtained after substituting $x=64y^2$ into (\ref{eq:dihb1}),
giving $z^3=y\,(1-64y^2)$. 
Here $x,y$ represent Hauptmoduln of  $\XX_1(2)$, $\XX(2)$ (respectively),
like in Remark \ref{rm:level5}.

Now we consider the tetrahedral case similarly.
With a Hauptmodul of $\Gamma_0(3)$ given in (\ref{eq:h3}),
the map $\XX_0(3)\rightarrow\XX(1)$ is 
\begin{equation}
j(\tau)=\frac{(h_3(\tau)+27)\,(h_3(\tau)+243)^3}{h_3(\tau)^3}.
\end{equation}
To match $1728/j(\tau)$ with $\varphi_3(x)$, we identify
\begin{equation}
x=-\frac{108}{h_3(\tau)+27}.
\end{equation}
Then $1+\frac14x=h_3(\tau)/(h_3(\tau)+27)$, and  formulas (\ref{eq:tetr1})--(\ref{eq:tetr2}) become
\begin{align}  \label{eq:teh1}
j(\tau)^{1/12}\,\hpg21{\!-\frac{1}{12},\frac1{4}}{\frac23}{\frac{1728}{j(\tau)}} 
& = h_3(\tau)^{-1/4}\,\big(h_3(\tau)+27\big)^{1/3}, \\
j(\tau)^{-1/4}\,\hpg21{\frac1{4},\,\frac{7}{12}}{\frac43}{\frac{1728}{j(\tau)}} 
& = h_3(\tau)^{-1/4}.
\end{align}
The function $h_3(\tau)^{-1/4}$ is invariant under the congruence subgroup 12B$^0$ in \cite{CumminsP},
with the branching pattern \mbox{$[12\!\cdot\!4/3^41^4/2^8]$} over the $j$-line. 
The quotient $(h_3(\tau)+27)^{1/3}$ 
is a Hauptmodul of $\XX(3)$, without a nice $q$-factorization apparently.
The function in (\ref{eq:teh1}) is invariant under the subgroup 12A$^3$ in \cite{CumminsP},
because it contains both $\Gamma(3)$ and 12B$^0$.  Its branching pattern is \mbox{$[12^4/3^{16}/2^{24}]$}.
The covering over $\XX(2)$ is obtained after substituting $x=-108y^3$ into (\ref{eq:tetr1}),
giving $z^4=y\,(1-27y^3)$. An intermediate congruence group between12A$^3$ and $\Gamma(3)$ is 6D$^1$, 
with the branching \mbox{$[6^4/3^8/2^{12}]$}.

In the octahedral case, the map $\XX_0(4)\rightarrow\XX(1)$ is 
\begin{equation}
j(\tau)=\frac{(h_4(\tau)^2+256h_4(\tau)+4096)^3}{h_4(\tau)^4\,(h_4(\tau)+16)}.
\end{equation}
To match $1728/j(\tau)$ with $\varphi_4(x)$, we identify
\begin{equation}
x=\frac{16}{h_4(\tau)+16}.
\end{equation}
A well-known classical identity \cite[(72)]{Zagier08} implies
\begin{equation}
h_4(\tau)+16=\frac{\eta(2\tau)^{24}}{\eta(4\tau)^{16}\eta(\tau)^8}
=\frac1q \, \prod_{n=1}^{\infty} \big(1+q^{n}\big)^{-8(-1)^{n}}.
\end{equation}
(In terms of \cite[\S 9]{Duke}, this function equals $16/u(\tau)^8$.)
Formulas (\ref{eq:octa1})--(\ref{eq:octa2}) become
\begin{align} \label{eq:octaj1}
j(\tau)^{1/24}\,\hpg21{\!-\frac1{24},\,\frac{7}{24}}{\frac34}{\frac{1728}{j(\tau)}} 
& =\frac{\eta(2\tau)^5}{\eta(4\tau)^2\,\eta(\tau)^3},\\
& = q^{-1/24} \, \prod_{n=1}^{\infty} \big(1+q^{2n-1}\big)^{\!3}\,\big(1+q^{2n}\big), 
\qquad\nonumber \\   \label{eq:octaj2}
j(\tau)^{-5/24}\;\hpg21{\frac5{24},\,\frac{13}{24}}{\frac54}{\frac{1728}{j(\tau)}} 
& =\frac{\eta(4\tau)^2}{\eta(2\tau)\,\eta(\tau)} \\
& = q^{5/24} \;\prod_{n=1}^{\infty} \big(1+q^{2n-1}\big)\,\big(1+q^{2n}\big)^{\!3}.  \nonumber
\end{align}
These functions are defined on the covering 
$z^6=y\,(1-16y^4)$ of $\XX(4)$, obtained by substituting $x=16y^4$ in  (\ref{eq:octa1}). 
The branching over $\XX(1)$ is $[24^6/3^{48}/2^{72}]$, of genus 10.
There are 5 congruence subgroups with this branching 
in \cite{CumminsP}. The correct one is 24A$^{10}$,
because it has intermediate supergroups with the branching patterns
 $[8^6/3^{16}/2^{24}]$ and $[12^6/3^{24}/2^{36}]$, namely, 8A$^2$ and 12A$^4$. 

It is fitting to mention here this classical identity \cite[(74)]{Zagier08}:
\begin{align}
\hpg21{\frac1{12},\,\frac{5}{12}}{1}{\frac{1728}{j(\tau)}} 
= E_4(\tau)^{1/4} = j(\tau)^{1/12}\,\eta(\tau)^2,
\end{align} 
where $E_4(\tau)=1+240q+2160q^2+\ldots$ is 
an Eisenstein series. 

\section{The cases (7A) and (7B)}
\label{sec:c277}

As explained in \S \ref{sec:ourcovs} 
the Darboux covering $\Phi_7:E_7\to \CC\PP^1$ with the branching pattern $[7^31^3/7^31^3/2^{12}]$
can be computed from $\Phi_3$ using cubic transformation (\ref{eq:t32c}). 

\subsection{Degree 24 map}
\label{sec:dc277}

The Darboux curve $E_7$ is determined as the fiber product
\begin{equation} \label{eq:fp77}
\Phi_3(x)=\frac{27z^2}{(4-z)^3}.
\end{equation}
To parametrize this curve by a simple equation, we substitute $z=8\hat{z}/(2\hat{z}+3)$
so that the right-hand side becomes $\hat{z}^{\,2}(2\hat{z}+3)$. 
After the next substitution $\hat{z}=\tilde{z}F_1^3/(G_0G_1)$
we get the equation 
\begin{equation}
1728x(x-1)F_1=\tilde{z}^2(2\tilde{z}F_1^3+3G_0G_1)
\end{equation}
of degree 12 in $x,\tilde{z}$. The curve defined by this equation
can be straightforwardly analyzed with {\sf Maple}'s standard package {\sf algcurves}. 
The curve has genus 1 and is isomorphic to (\ref{eq:darbc7}). 
Eventually, we can find this parametrization of (\ref{eq:fp77})
by the elliptic curve  (\ref{eq:darbc7}): 
\begin{equation}
x=\frac{4(u+v)^2}{(4u-1)^2(8u-1)}, \qquad z=\Phi_7,
\end{equation}
where 
\begin{equation}
\Phi_7= \frac{128(1-4u)(-v-3u+4uv+20u^2)^7}{u^3(1-8u)(1-4v-20u+64u^2)^7}.
\end{equation}
This is a Belyi map with the anticipated branching pattern $[7^31^3/7^31^3/2^{12}]$. 
Note that $\Phi_7$ vanishes at $(u,v)=(0,0)\in E_7$ despite the factor $u^3$ in the denominator.
The divisor of $\Phi_7$ is
\begin{align}
\mbox{div}(\Phi_7)= &\, \textstyle (0,0)+(\frac14,\frac14)+(\frac14,-\frac14)+7U_1+7U_2+7U_3 
\nonumber \\
& \textstyle -{\cal O}-(\frac18,\frac18)-(\frac18,-\frac18)-7V_1-7V_2-7V_3.
\end{align}
Here the $u$-coordinates of $U_1,U_2,U_3$ satisfy $4u(4u-1)(4u-5)=1$, 
and the $u$-coordinates of $V_1,V_2,V_3$ satisfy $16u(4u-1)(8u-3)=1$.  
Table \ref{t:div724} presents several straightforward  rational functions on $E_7$
and their divisors. We use them to present rational and power functions on $E_7$.
For example,
\begin{equation}
\Phi_7 = -\frac{128(1-4u)\widehat{G}_3^{\,7}}{u^3(1-8u)G_4^{\,7}}
= -\frac{128u^4(1-4u)G_3^{\,7}}{(1-8u)\widehat{G}_4^{\,7}}.
\end{equation}

\begin{table}[t!]
\begin{center}
\begin{tabular}{|c|c|c|}
\hline
Id & Function & Divisor \\ 
\hline
--- & $u$ & $2 (0,0)-2{\cal O}$ \\[1pt]
--- & $1-4u$ & $(\frac14,\frac14)+(\frac14,-\frac14)-2{\cal O}$ \\[2pt]
--- & $1-8u$ & $(\frac18,\frac18)+(\frac18,-\frac18)-2{\cal O}$ \\[2pt]
--- & $v-u$ & $(0,0)+(\frac14,\frac14)+(\frac18,\frac18)-3{\cal O}$ \\[2pt]
--- & $v+u$ & $(0,0)+(\frac14,-\frac14)+(\frac18,-\frac18)-3{\cal O}$ \\[2pt]
$F_3$ & $1-4v-4u$ & $3(\frac18,\frac18)-3{\cal O}$ \\[2pt]
$\widetilde{F}_3$  & $1+4v-4u$ & $3(\frac18,-\frac18)-3{\cal O}$ \\[2pt]
$F_4$ & $1-2v-6u$ & $2(\frac14,-\frac14)+(\frac18,\frac18)-3{\cal O}$ \\[2pt]
$\widetilde{F}_4$ & $1+2v-6u$ & $2(\frac14,\frac14)+(\frac18,-\frac18)-3{\cal O}$ \\[2pt]
$G_3$ & $1+2v-10u+16u^2$ & $U_1+U_2+U_3+(\frac14,\frac14)-4{\cal O}$ \\[2pt]
$G_4$ & $1-4v-20u+64u^2$ & $V_1+V_2+V_3+(\frac18,-\frac18)-4{\cal O}$ \\[2pt]
$\widehat{G}_3$& $v+3u-4uv-20u^2$ & $U_1+U_2+U_3+(0,0)+(\frac18,-\frac18)-5{\cal O}$ \\[2pt]
 $\widehat{G}_4$& $v-5u-8uv+24u^2$ & $V_1+V_2+V_3+(0,0)+(\frac14,\frac14)-5{\cal O}$ \\[1pt]
 \hline
\end{tabular}
\end{center}
\caption{Divisors on the curve $v^2=u(1-11u+32u^2)$}
\label{t:div724}
\end{table}

\subsection{Evaluations of type (7A)}

Computation of Darboux evaluations on a genus 1 curve (such as $E_7$)
is less straightforward than on $\PP^1$, because divisors of rational or radical functions 
are restricted by the Picard group \cite[p.~28]{silverman}, 
or the group law on an elliptic curve such as (\ref{eq:darbc7}). 
Consequently, radical factors have to be routinely chosen with extraneous zeroes or poles,
and compensatory factors are often needed. 
A practical demonstration of involved combinatorial calculations 
is given in \cite[\S 4]{ViDarb}.
The obtained Darboux evaluations can be checked by expanding the power series 
in $\sqrt{u}$ on both sides.

\begin{theorem} \label{th:de7a}
With reference to Table $\ref{t:div724}$, we have
\begin{align} \label{eq:de7a1} 
\hpg32{\!-\frac1{14},\,\frac1{14},\,\frac5{14}}{\frac17,\;\frac57}{\Phi_7}
= & \, \frac{F_3^{\,1/14}\,F_4^{\,1/7}\,\widetilde{F}_4^{\,3/7}}{\sqrt{G_4}},\\ 
\hpg32{\frac3{14},\,\frac5{14},\,\frac9{14}}{\frac37,\;\frac97}{\Phi_7}
= & \, \frac{(1-4u)^{4/7}\,F_3^{\,1/14}\,\widetilde{F}_4^{\,4/7}\,G_4^{\,3/2}}
{\widetilde{F}_3^{\,4/7}\,G_3^{\,2}},\\ \label{eq:de7a3}
\hpg32{\frac{11}{14},\,\frac{13}{14},\,\frac{17}{14}}{\frac{11}7,\;\frac{13}7}{\Phi_7}
= & \, \frac{(1-8u)\,\widetilde{F}_3^{\,5/14}\,\widehat{G}_4^{\,11/2}}
{u^{16/7}\,(v-u)^{1/14}\,(v+u)^{6/7}\,G_3^{\,6}}.
\end{align}
\end{theorem}
\begin{proof}
Preliminary expressions of the $\hpgo32$-functions 
in terms of (complicated) radical functions on $E_7$ can be obtained
by applying cubic transformation (\ref{eq:t32c}) to the functions in Theorem \ref{eq:th237a}.
Then we compute fractional divisors 
of the $\hpgo32$-functions in (\ref{eq:de7a1})--(\ref{eq:de7a3}), respectively:
\begin{align}
& \textstyle 
\frac{6}{7} (\frac14,\frac14)+\frac2{7} (\frac14,-\frac14)
+\frac{5}{14} (\frac18,\frac18)+\frac1{14}{\cal O}-\frac1{14} (\frac18,-\frac18)
-\frac12V_1-\frac12V_2-\frac12V_3,\\
& \textstyle \frac4{7} (\frac14,-\frac14)- \frac{2}{7} (\frac14,\frac14)-2U_1-2U_2-2U_3 \\
& \hspace{83pt} \textstyle
+\frac9{14}{\cal O}+\frac5{14} (\frac18,-\frac18)+\frac{3}{14} (\frac18,\frac18)
+\frac32V_1+\frac32V_2+\frac32V_3, \notag \\  \label{eq:divf7a3}
& \textstyle\! -\frac{4}{7} (\frac14,\frac14)-\frac6{7} (\frac14,-\frac14)-6U_1-6U_2-6U_3\\
& \hspace{83pt} \textstyle
+ \frac{17}{14} (\frac18,-\frac18)+\frac{13}{14} (\frac18,\frac18)+\frac{11}{14}{\cal O}
+\frac{11}2V_1+\frac{11}2V_2+\frac{11}2V_3.  \qquad \notag
\end{align}
The radical expressions on the right-hand sides of (\ref{eq:de7a1})--(\ref{eq:de7a3}) 
have the same divisors, and evaluate to $1$ at $(u,v)=(0,0)$.

For a constructive proof, we seek to combine (multiplicatively) the functions in Table $\ref{t:div724}$
to obtain the required divisors. The routine is demonstrated in \cite[\S 4.2--4.3]{ViDarb}. 
One strategy could be: to produce the required vanishing orders of the $U_i$'s and $V_i$'s
by using $G_3,G_4$;  then use $1-4u,1-8u,\widetilde{F}_4,\widetilde{F}_3$ consequently
to account for the vanishing orders of 
$(\frac14,-\frac14)$, $(\frac18,\frac18)$, $(\frac14,\frac14)$, $(\frac18,-\frac18)$, respectively.
For example, this produces
\begin{equation}
 \frac{G_4^{11/2}\,(1-8u)^{13/14}\,\widetilde{F}_4^{\,22/7}}
 {G_3^{\,6}\;(1-4u)^{6/7}\;\widetilde{F}_3^{\,39/14}}
\end{equation}
for divisor (\ref{eq:divf7a3}). 
The powers of $\widetilde{F}_3$, $\widetilde{F}_4$ are rather awkward. 
We can modify this expression by the identities
\begin{align}
(v-u)G_4 & =(1-8u)\widehat{G}_4, & (1-4u)(v-u)\widetilde{F}_3&=(1-8u)(v+u)\widetilde{F}_4,
 \qquad\notag \\
(v+u)G_3 & =(1-4u)\widehat{G}_3,  & (1-4u)(v+u)F_3 & =(1-8u)(v-u)F_4, \\
F_4\widetilde{F}_4& =(1-4u)^2(1-8u), \hspace{-8pt} &  (v-u)(v+u) & =u(1-4u)(1-8u) \notag
\end{align}
to get a more polished formula like (\ref{eq:de7a3}).
\end{proof}

The companion hypergeometric solutions at $\Phi_7=\infty$ are
\begin{align}
& \Phi_7^{\,1/14}\;\hpg32{\!-\frac1{14},\,\frac3{14},\,\frac{11}{14}}{\frac47,\;\frac67}{\frac1{\Phi_7}},\qquad
\Phi_7^{-1/14}\;\hpg32{\frac1{14},\,\frac5{14},\,\frac{13}{14}}{\frac57,\;\frac87}{\frac1{\Phi_7}}, \quad\notag \\
& \Phi_7^{-5/14}\;\hpg32{\frac{5}{14},\,\frac{9}{14},\,\frac{17}{14}}{\frac{9}7,\;\frac{10}7}{\frac1{\Phi_7}}.
\end{align}
Apart from the power factors and the argument $1/\Phi_7$,
they are contiguous to the $\hpgo32$-functions in (\ref{eq:de7a1})--(\ref{eq:de7a3}).
The functions $\Phi_7$ and $1/\Phi_7$ can be interchanged by this automorphism of $E_7$:
\begin{equation} \label{eq:ec7inv}
(u,v)\mapsto \left(\frac1{32u},-\frac{v}{32u^2} \right).
\end{equation}
\begin{theorem}
With reference to Table $\ref{t:div724}$, we have
\begin{align} \label{eq:de7c1}
\hpg32{\!-\frac1{14},\,\frac3{14},\,\frac{11}{14}}{\frac47,\;\frac67}{\Phi_7}
= & \, \frac{(1-4u)^{1/7}\,\widetilde{F}_3^{\,3/14}\,\widetilde{F}_4^{\,1/7}}{(1-8u)^{1/14}\;\sqrt{G_4}},\\
\hpg32{\frac1{14},\,\frac5{14},\,\frac{13}{14}}{\frac57,\;\frac87}{\Phi_7}
= & \, \frac{(1-4u)^{2/7}(1-8u)^{1/7}\,F_3^{\,1/14}\,\widetilde{F}_4^{\,2/7}\sqrt{G_4}}{G_3},\\
\label{eq:de7c3}
\hpg32{\frac{5}{14},\,\frac{9}{14},\,\frac{17}{14}}{\frac{9}7,\;\frac{10}7}{\Phi_7}
= & \, \frac{(1-8u)\,(v-u)^{3/14}\,\widehat{G}_4^{\,5/2}}
{u^{8/7}\,(v+u)^{3/7}\,\widetilde{F}_3^{\,1/14}\,G_3^{\,3}}.
\end{align}
\end{theorem}
\begin{proof}
We compute preliminary algebraic expressions on $E_7$ for the $\hpgo32$-functions
by either of these two ways: 
\begin{itemize}
\item by using the mentioned contiguous relations;
\item or by applying cubic transformation (\ref{eq:t32b}) with $z=1/\Phi_7$
to the functions in Theorem \ref{eq:th237a},  and subsequently interchanging 
$1/\Phi_7\mapsto \Phi_7$ by (\ref{eq:ec7inv}).
\end{itemize}
Then we compute their fractional divisors, respectively:
\begin{align}
& \textstyle 
\frac{3}{7} (\frac14,\frac14)+\frac1{7} (\frac14,-\frac14)
+\frac{11}{14}{\cal O}+\frac{3}{14} (\frac18,-\frac18)-\frac1{14} (\frac18,\frac18)
-\frac12V_1-\frac12V_2-\frac12V_3,\\
& \textstyle \frac2{7} (\frac14,-\frac14)- \frac{1}{7} (\frac14,\frac14)-U_1-U_2-U_3 \\
& \hspace{83pt} \textstyle
+\frac{13}{14}(\frac18,-\frac18)+\frac5{14}(\frac18,\frac18) +\frac{1}{14} {\cal O}
+\frac12V_1+\frac12V_2+\frac12V_3, \nonumber \\
& \textstyle\! -\frac{2}{7} (\frac14,\frac14)-\frac3{7} (\frac14,-\frac14)-3U_1-3U_2-3U_3\\
& \hspace{83pt} \textstyle
+ \frac{17}{14} (\frac18,\frac18)+\frac{9}{14} {\cal O}+\frac{5}{14}(\frac18,-\frac18)
+\frac{5}2V_1+\frac{5}2V_2+\frac{5}2V_3. \nonumber
\end{align}
The radical expressions on the right-hand sides of (\ref{eq:de7c1})--(\ref{eq:de7c3}) 
have the same divisors, and evaluate to $1$ at $(u,v)=(0,0)$.
\end{proof}

\subsection{Evaluations of type (7B)}

Similarly, we can apply cubic transformation (\ref{eq:t32c}) to the formulas of Theorem \ref{eq:th237b}
and eventually obtain formulas like
\begin{align} \label{eq:de7bp}
\hpg32{\!-\frac3{14},\,\frac1{14},\,\frac{3}{14}}{\frac17,\;\frac37}{\Phi_7}
= & \, \frac{(1-4u)^{4/7}\,F_4^{\,1/7}\,(1+2v-2u+32uv)}{(1-8u)^{1/14}\,\widetilde{F}_3^{\,3/14}\,G_4^{3/2}},
\end{align}
with a few zeroes at regular points of the pulled-back Fuchsian equation. 
In this example, the $u$-coordinates of those regular points satisfy $8u(64u^2+2u-1)=1$.
But there are $\hpgo32$-functions of type (7B) that do not vanish at regular points.

\begin{theorem} \label{th:de7b}
With reference to Table $\ref{t:div724}$, we have
\begin{align} \label{eq:de7b1}
\hpg32{\!-\frac1{14},\,\frac1{14},\,\frac{9}{14}}{\frac27,\;\frac67}{\Phi_7}
= & \, \frac{(1-4u)^{1/7}\,1-8u)^{4/7}\,F_4^{\,2/7}}{F_3^{\,1/14}\;\sqrt{G_4}},\\
\label{eq:de7b2}
\hpg32{\frac1{14},\,\frac3{14},\,\frac{11}{14}}{\frac37,\;\frac87}{\Phi_7}
= & \, \frac{(1-8u)^{1/14}\,\widetilde{F}_4^{\,6/7}\sqrt{G_4}}{(1-4u)^{1/7}\,\widetilde{F}_3^{\,3/14}\,G_3},\\
\label{eq:de7b3}
\hpg32{\frac{9}{14},\,\frac{11}{14},\,\frac{19}{14}}{\frac{11}7,\;\frac{12}7}{\Phi_7}
= & \, \frac{(1-8u)\,\widetilde{F}_3^{\,1/14}\,\widehat{G}_4^{\,9/2}}
{u^{13/7}\,(v-u)^{3/14}\,(v+u)^{4/7}\,G_3^{\,5}}.
\end{align}
\end{theorem}
\begin{proof}
Preliminary expressions on $E_7$ for a set of contiguous $\hpgo32$-functions (such as (\ref{eq:de7bp}))
can be obtained by applying cubic transformation (\ref{eq:t32c}) to the formulas of Theorem \ref{eq:th237b}.
By using contiguous relations we then compute preliminary expressions for the target 
functions. Their fractional divisors are 
\begin{align}
& \textstyle  \label{eq:de7bfd1}
\frac5{7} (\frac14,-\frac14)+\frac{1}{7} (\frac14,\frac14)
+\frac{9}{14} (\frac18,\frac18)+\frac{1}{14} (\frac18,-\frac18)-\frac1{14} {\cal O}
-\frac12V_1-\frac12V_2-\frac12V_3,\\ \label{eq:de7bfd2} 
& \textstyle \frac4{7} (\frac14,\frac14)-\frac{1}{7} (\frac14,-\frac14)-U_1-U_2-U_3 \\
& \hspace{83pt} \textstyle
+\frac{11}{14}(\frac18,-\frac18)+\frac3{14} {\cal O}+\frac{1}{14} (\frac18,\frac18) 
+\frac12V_1+\frac12V_2+\frac12V_3, \notag \\
& \textstyle\! -\frac{4}{7} (\frac14,-\frac14)-\frac5{7} (\frac14,\frac14)-5U_1-5U_2-5U_3\\
& \hspace{83pt} \textstyle
+ \frac{19}{14} {\cal O}+\frac{11}{14} (\frac18,\frac18)+\frac{9}{14}(\frac18,-\frac18)
+\frac{9}2V_1+\frac{9}2V_2+\frac{9}2V_3. \notag
\end{align}
The radical expressions on the right-hand sides of (\ref{eq:de7b1})--(\ref{eq:de7b3}) 
have the same divisors, and evaluate to $1$ at $(u,v)=(0,0)$.

Alternatively, we can follow the strategy in the proof of Theorem \ref{eq:th237a}. 
In (\ref{eq:kl7blb}) we see a basis of local solutions at $z=0$ 
of a hypergeometric equation of type (7B). 
For a radical solution with the local exponent $0$ at $(u,v)=(0,0)$, we have 12 candidate
fractional divisors in total: apart from (\ref{eq:de7bfd1}) we can permute the exponents 
to $(\frac14,-\frac14)$ and $(\frac14,\frac14)$, and the exponents to $\cal O$, $(\frac18,\frac18)$,
$(\frac18,-\frac18)$. Starting from a preliminary expression of one candidate solution,
preliminary expressions for other candidates can be obtained by using the factor 
$F_4^{\,2/7}\widetilde{F}_4^{-2/7}F_3^{-2/21}\widetilde{F}_3^{\,2/21}$ to permute the exponents of
 $(\frac14,-\frac14)$ and $(\frac14,\frac14)$, and so on. 
By checking all 12 candidates we find the right solution with the divisor (\ref{eq:de7bfd1}).
By Lemma \ref{th:locexps}, we have two candidates for the transformed 
second solution in (\ref{eq:kl7blb}) with the exponent $1/7$ at $(u,v)=(0,0)$. 
Their divisors differ by a cyclic permutation of the exponents to $\cal O$, $(\frac18,\frac18)$,
$(\frac18,-\frac18)$. The right solution leads to formula (\ref{eq:de7b2}).
For the third solution, we have just one candidate. 
\end{proof}

The companion hypergeometric solutions at $\Phi_7=\infty$ are
\begin{align}
& \Phi_7^{\,1/14} \; \hpg32{\!-\frac1{14},\,\frac1{14},\,\frac{9}{14}}{\frac27,\;\frac67}{\frac1{\Phi_7}}\!, 
 \qquad
\Phi_7^{-1/14} \; \hpg32{\frac1{14},\,\frac3{14},\,\frac{11}{14}}{\frac37,\;\frac87}{\frac1{\Phi_7}}\!,
 \quad\notag \\
& \Phi_7^{-9/14}\; \hpg32{\frac{9}{14},\,\frac{11}{14},\,\frac{19}{14}}{\frac{11}7,\;\frac{12}7}{\frac1{\Phi_7}}\!.
\end{align}
The parameters of the $\hpgo32$-solutions are the same as in Theorem \ref{th:de7b},
and the argument can be transformed to $\Phi_7$ by (\ref{eq:ec7inv}).
There is thus no need for separate formulas.  
The symmetry of parameters is reflected by the symmetric matrix in (\ref{eq:3f2mat}).
The symmetry is a characteristic of $\hpgo32$-functions on the right hand-side of the 
quadratic transformation (\ref{eq:qua47}).

\section{The cases (4A) and (4B)}
\label{sec:c247}

As explained in \S \ref{sec:ourcovs}, 
the Darboux covering \mbox{$\Phi_4:E_4\to \CC\PP^1$} with the branching pattern $[7^31^3/4^6/2^{12}]$
can be computed from $\Phi_7$ using quadratic transformation (\ref{eq:qua47}).
The covering $\Phi_4$ is given by formula (\ref{eq:qua47}), but the curve $E_4$ is defined
by the involution of (\ref{eq:ec7inv}) of $E_7$. The involution-invariant functions
\begin{equation}
\tilde{p}=u+\frac{1}{32u},\qquad  \widetilde{w}=v-\frac{v}{32u^2}
\end{equation}
satisfy $8\tilde{w}^2=(32\tilde{p}-11)(8\tilde{p}^2-1)$. Taking
\begin{align}
p=&\quad\! \frac{1}{32\tilde{p}-11} \hspace{7pt} = \;\, \frac{u}{1-11u+32u^2}, \\
w=&\; \frac{32\tilde{w}}{(32\tilde{p}-11)^2} = \frac{v\,(1-32u^2)}{(1-11u+32u^2)^2}
\end{align}
we get the elliptic curve (\ref{eq:darbc7}).
The functions $(p,w)$ define a 2-isogeny \cite[\S III.4]{silverman} 
between $E_4$ and $E_7$ as elliptic curves. 

The divisor of $\Phi_4$ on $E_4$ is computed to be
\begin{align}
(0,0)+(1,4)+(1,-4)+7S_1+7S_2+7S_3 \hspace{90pt} \nonumber \\
-4T_1-4T_2-4T_3-4T_4-4T_5-4T_6.
\end{align}
Here the $p$-coordinates of $S_1,S_2,S_3$ satisfy $7p(7p^2-21p+5)=1$,
while the $p$-coordinates of $T_1,\ldots,T_6$ satisfy 
\begin{equation} \label{eq:teq}
(49p^2-29)(49p^3+98p^2-188p-10)+435p=\frac{1}{7p}.
\end{equation} 
Table \ref{t:phys2} presents several 
rational functions on $E_4$ and their divisors. 
We use them to present rational and power functions on $E_4$.
In particular,
\begin{equation}
\Phi_4=\frac{512\,(w-4p)\,F_5^{\,7}}{(1+w+3p)\,G_5^{\,4}}.
\end{equation}

It is more straightforward to compute representative Darboux evaluations of type (4B),
because of a direct relation to type (7B) evaluations by quadratic transformation (\ref{eq:qua47}).
\begin{theorem}
With reference to Table $\ref{t:phys2}$, we have
\begin{align} \label{eq:de4b1}
\hpg32{\!-\frac1{28},\,\frac3{14},\,\frac{13}{28}}{\frac27,\;\frac67}{\Phi_4}
= & \, \frac{(1-p)^{2/7}\,F_6^{\,1/14}}{G_5^{\,1/7}},\\
\hpg32{\frac3{28},\,\frac5{14},\,\frac{17}{28}}{\frac37,\;\frac87}{\Phi_4}
= & \, \frac{\widetilde{F}_6^{\,3/14}\,G_5^{\,3/7}}
{(1-p)^{5/7}\,F_5},\\  \label{eq:de4b3}
\hpg32{\frac{19}{28},\,\frac{13}{14},\,\frac{33}{28}}{\frac{11}7,\;\frac{12}7}{\Phi_4}
= & \, \frac{\widetilde{F}_6^{\,5/14}\,G_5^{\,19/7}}
{(1-p)^{4/7}\,F_5^{\,5}}.
\end{align}
\end{theorem}
\begin{proof}
Preliminary expressions of the $\hpgo32$-functions 
in terms of radical functions on $E_4$ can be obtained
by applying quadratic transformation (\ref{eq:qua47}) to the functions in Theorem \ref{th:de7b}.
Their fractional divisors, respectively:
\begin{align}
& \textstyle \frac57(1,4)+\frac17(1,-4)-\frac17(T_1+\ldots+T_6),\\
& \textstyle \frac47(1,-4)-\frac17(1,4)-S_1-S_2-S_3+\frac37(T_1+\ldots+T_6),\\
& \textstyle -\frac47(1,4)-\frac57(1,-4)-5S_1-5S_2-5S_3+\frac{19}7(T_1+\ldots+T_6).
\end{align}
The radical expressions on the right-hand sides of (\ref{eq:de4b1})--(\ref{eq:de4b3}) 
have the same divisors, and evaluate to $1$ at $(p,w)=(0,0)$.
\end{proof}

\begin{table}[t!]
\begin{center}
\begin{tabular}{|c|c|c|}
\hline Id & function & divisor \\ 
\hline
--- & $p$ & $2 (0,0)-2{\cal O}$ \\[2pt]
--- & $1-p$ & $(1,4)+(1,-4)-2{\cal O}$ \\[2pt]
--- & $w-4p$ & $(0,0)+(1,4)+(-\frac17,-\frac47)-3{\cal O}$ \\[2pt]
--- & $w+5p-p^2$ & $(0,0)+3(1,-4)-4{\cal O}$ \\[2pt]
--- & $1-w+3p$ & $2(1,4)+(-\frac17,\frac47)-3{\cal O}$ \\[2pt]
--- & $1+w+3p$ & $2(1,-4)+(-\frac17,-\frac47)-3{\cal O}$ \\[2pt]
--- & $1+7w+35p$ & $3(-\frac17,\frac47)-3{\cal O}$ \\[2pt]
--- & $1-7w+35p$ & $3(-\frac17,-\frac47)-3{\cal O}$ \\[2pt]
$F_5$ & $1-2w-16p+7p^2$ & $(1,-4)+S_1+S_2+S_3-4{\cal O}$ \\[2pt]
$F_6$ & $1-10w+47p+2wp-17p^2+p^3$ & $6(1,4)-6{\cal O}$ \\[2pt]
$\widetilde{F}_6$ & $1+10w+47p-2wp-17p^2+p^3$ & $6(1,-4)-6{\cal O}$ \\[2pt]
$G_5$ & \multicolumn{1}{l|}{$1+47w+89p-14wp$} & \\
&  \multicolumn{1}{r|}{$+91p^2-49wp^2-245p^3$} & $(1,-4)+T_1+\ldots+T_6-7{\cal O}$ \\
\hline
\end{tabular}
\end{center}
\caption{Divisors on the curve $w^2=p(1+22p-7p^2)$}
\label{t:phys2}
\end{table}

The companion hypergeometric solutions at $\Phi_4=\infty$ are
\begin{align}
& \Phi_4^{\,1/28} \; \hpg32{\!-\frac1{28},\,\frac3{28},\,\frac{19}{28}}{\frac12,\;\frac34}{\frac1{\Phi_4}}\!, 
 \qquad
\Phi_4^{-3/14} \; \hpg32{\frac3{14},\,\frac5{14},\,\frac{13}{14}}{\frac34,\;\frac54}{\frac1{\Phi_4}}\!,
 \quad\notag \\
& \Phi_4^{-13/28}\; \hpg32{\frac{13}{28},\,\frac{17}{28},\,\frac{33}{28}}
{\frac{5}4,\;\frac{3}2}{\frac1{\Phi_4}}\!.
\end{align}
Their Darboux evaluations would hold around one of the points $T_1,\ldots,T_6$,
defined by (\ref{eq:teq}) and $G_5=0$.
All those points are defined over $\QQ(\sin\frac{\pi}7)$.
Handy Darboux evaluations of degree 21 for type (4B) are given in \cite[\S 4.4]{Vid18b}.

Darboux evaluations of type (4A) always have zeroes at some regular points, apparently.
Determining those zeroes 
can be tricky.
\begin{theorem}
With reference to Table $\ref{t:phys2}$, we have
\begin{align} \label{eq:de4a1}
\hpg32{\!-\frac3{28},\,\frac{11}{28},\,\frac9{14}}{\frac47,\;\frac67}{\Phi_4}
= & \, \frac{(1-7w-21p)\,(1-p)^{6/7}\,F_6^{\,3/14}}{(1-w+3p)\,G_5^{\,3/7}},\\
\hpg32{\frac1{28},\,\frac{15}{28},\,\frac{11}{14}}{\frac57,\;\frac87}{\Phi_4}
= & \, \frac{\big(1-\frac73p\big)\,(1-p)^{2/7}\;\widetilde{F}_6^{\,1/14}\,G_5^{\,1/7}}
{F_5},\\ \label{eq:de4a3}
\hpg32{\frac{9}{28},\,\frac{23}{28},\,\frac{15}{14}}{\frac{9}7,\;\frac{10}7}{\Phi_4}
= & \, \frac{(1+7w-21p)\,\sqrt{p}\;(1-p)^{4/7}\,\widetilde{F}_6^{\,1/7}\,G_5^{\,9/7}}
{(q-4p)\,F_5^{\,3}}.
\end{align}
\end{theorem}
\begin{proof}
We can apply quadratic transformation (\ref{eq:qua47}) to these hypergeometric functions
and obtain $\hpgo32$-functions that are contiguous to the formulas of Theorem \ref{th:de7a}.
This allows us to derive preliminary expressions of (\ref{eq:de4a1})--(\ref{eq:de4a3}) 
as radical functions on $E_4$,
and their fractional divisors (with known points $R_1,R_2$, etc):
\begin{align} \label{eq:de4da1}
& \textstyle \frac37(1,-4)+\frac17(1,4)+R_1+R_2-\frac37(T_1+\ldots+T_6),\\
\label{eq:de4da2}
& \textstyle \frac27(1,4)-\frac17(1,-4)+\widetilde{R}_1+\widetilde{R}_2
-S_1-S_2-S_3+\frac17(T_1+\ldots+T_6),\\
\label{eq:de4da3}
& \textstyle\! -\frac27(1,-4)-\frac37(1,4)+\widehat{R}_1+\widehat{R}_2
-3S_1-3S_2-3S_3+\frac{9}7(T_1+\ldots+T_6).
\end{align}
The radical expressions on the right-hand sides of (\ref{eq:de4a1})--(\ref{eq:de4a3}) 
have the same divisors, and evaluate to $1$ at $(p,w)=(0,0)$.

Alternatively, we can follow the strategy in the proofs of Theorems \ref{eq:th237a} and \ref{th:de7b}.
For the $\hpgo32$-function in (\ref{eq:de4a1}), the points $R_1,R_2$ in the divisor (\ref{eq:de4da1})
are unknown {\em a priori}, and besides, there is other possible divisor shape 
with the $\QQ$-coefficients to $(1,-4)$ and $(1,4)$ interchanged. 
In both cases, the line in $\CC\PP^2$ through $R_1,R_2$ intersects $E_4$ at the third point
that must be a torsion point defined over $\QQ$.  
We see the point $(p,w)=(1,4)$ of order 6, and there is only one 2-torsion point $(0,0)$. 
By Mazur's theorem \cite[Theorem 7.5]{silverman}, 
the torsion subgroup is then either $\ZZ/6\ZZ$ or $\ZZ/12\ZZ$.
Existence of rational 4-torsion points distinguishes these two cases. The tangent lines 
to $E_4$ at  the 4-torsion points have the form $y=\alpha x$. The  tangent lines of this form
satisfy $\alpha^4-44\alpha^2+512=0$. Hence there are no rational $4$-torsion points,
and the $\QQ$-rational torsion is $\ZZ/6\ZZ$.
The line through undetermined $R_1,R_2$ is a line through one of the 6 torsion points.
Among the possibilities (adjusted by a compensatory factor vanishing 
only at the supposed torsion point), we find the radical function expressed in ({eq:de4a1}).
For the other two hypergeometric functions, the divisors have the shapes
(\ref{eq:de4da2}) and (\ref{eq:de4da3}) by Lemma (\ref{th:locexps}).
The points $\widetilde{R}_1,\widetilde{R}_2$ or $\widehat{R}_1,\widehat{R}_2$
are determined by trying out lines through the 6 torsion points. 
\end{proof}

\small 

\bibliographystyle{alpha}

\end{document}